\documentclass{amsart}

\usepackage{graphicx}
\usepackage[utf8]{inputenc}
\usepackage{nicefrac}
\usepackage{mathabx}
\usepackage{amsrefs}
\usepackage{esint}
\usepackage{mathtools}
\mathtoolsset{showonlyrefs}
\usepackage{wrapfig}

\newtheorem{theorem}{Theorem}[section]
\newtheorem{lemma}[theorem]{Lemma}
\newtheorem{mainthm}[theorem]{Main Theorem}

\theoremstyle{definition}
\newtheorem{definition}[theorem]{Definition}

\newtheorem{prop}[theorem]{Proposition}
\newtheorem{cor}[theorem]{Corollary}

\allowdisplaybreaks
\theoremstyle{remark}
\newtheorem{remark}[theorem]{Remark}

\numberwithin{equation}{section}

\allowdisplaybreaks

\newcommand{\dx}{\;\mathrm{d}x}
\newcommand{\dy}{\;\mathrm{d}y}

\newcommand{\ds}{\;\mathrm{d}s}
\newcommand{\dz}{\;\mathrm{d}z}

\newcommand{\mres}{\mathbin{\vrule height 1.6ex depth 0pt width
0.13ex\vrule height 0.13ex depth 0pt width 1.3ex}}

\makeatletter
\@namedef{subjclassname@2020}{%
  \textup{2020} Mathematics Subject Classification}
\makeatother

\begin{document}

\title[Poisson equation involving $\mathcal{H}^{n-1}\mres \Gamma$]{The Poisson equation involving surface measures}

\author{Marius Müller}
\address{Albert-Ludwigs-Unversität Freiburg, Mathematisches Institut, 79104 Freiburg im Breisgau}
\email{marius.mueller@math.uni-freiburg.de}
\thanks{The author would like to thank Jan Metsch and Ernst Kuwert for helpful discussions.}
%





\begin{abstract}
We prove the (optimal) $W^{1,\infty}$-regularity of weak solutions to the equation $-\Delta u = Q \; \mathcal{H}^{n-1}\mres \Gamma$ in a domain $\Omega \subset \mathbb{R}^n$ with Dirichlet boundary conditions, where $\Gamma \subset \subset \Omega$ is a compact (Lipschitz) manifold and $Q \in L^\infty(\Gamma)$.

We also discuss optimality  and necessity of the assumptions on $Q$ and $\Gamma$. 

Our findings can be applied to study the regularity of solutions for several free boundary problems, in particular the biharmonic Alt-Caffarelli Problem.
\end{abstract}
\maketitle

\section{Introduction}

In this article we study the Poisson equation with measure-valued right hand side of the form
\begin{equation}
\begin{cases} \label{eq:measspec}
- \Delta v = Q \; \mathcal{H}^{n-1}\mres \Gamma  & in \; \Omega, \\ 
\quad \; \;  v = 0 & on \; \partial \Omega,
\end{cases} 
\end{equation}
for a (smooth) domain $\Omega$, a (suitably regular) surface $\Gamma\subset \subset \Omega$, and $Q \in L^1(\Gamma,\mathcal{H}^{n-1})$.
The notion of (weak) solutions for problems of the form
\begin{equation}\label{eq:measintro}
\begin{cases} 
- \Delta v = \mu & in \; \Omega, \\ 
 \quad \; \; v = 0 & on \; \partial \Omega,
\end{cases} 
\end{equation}
for a (signed) measure $\mu$ is established in the literature and goes back to \cite{Stampacchia}. Regularity theory is well-established in various contexts (cf. \cite{Duzaar1}, \cite{Duzaar2}, \cite{Mingione}, \cite{Xiong}). This is closely tied to the study of the equation $-\Delta u = \mathrm{div}(F)$ via the \emph{duality method}. The latter equation is subject to extensive research with wide applications, e.g. to the \emph{Helmholtz decomposition}.

We are interested in the optimal regularity of (weak) solutions of \eqref{eq:measspec}. 

It turns out that general regularity results are either not applicable or not optimal if $\mu$ has the special structure $\mu = Q \; \mathcal{H}^{n-1} \mres \Gamma$. While the standard {duality method} (cf. Section \ref{sec:dualmeth}, \cite{Breit}, \cite{Breit2}) or the method of \emph{layer potentials}
(cf. Section \ref{sec:35}, \cite[Section 7.11]{Taylor}, \cite[Section 14]{Miranda})
can be used to establish the $W^{1,p}$-regularity for all $p \in [1, \infty)$, both methods have limitations in the case of  $p = \infty$. 
The question of $W^{1,\infty}$-regularity (or equivalently $C^{0,1}$-regularity) has already been raised in \cite[p.171]{Kilpelainen}, where it is stated as an open problem. 

In this article we show $W^{1,\infty}$-regularity in the case that $\Gamma$ is a compact Lipschitz manifold and $Q \in L^\infty( \Gamma)$. We also prove that the assumption  $Q \in L^\infty(\Gamma)$  and the Lipschitz regularity of $\Gamma$ are necessary and discuss in which sense $W^{1,\infty}$-regularity is optimal.


The structure $\mu = Q \; \mathcal{H}^{n-1} \mres \Gamma$  is in many ways a critical limit case of the established theory.

One way to see this is to look at the \emph{potential}
\begin{equation}
    U(x) := \int F(x-y) \; \mathrm{d}\mathcal{\mu}(y),
\end{equation}
where $F$ is the fundamental solution to the Laplace equation. It is well-known that $|\nabla_x F(x-y)| = \frac{C_n}{|x-y|^{n-1}}$, i.e. it has a singularity of polynomial order $n-1$ at $y = x$.  If we take  $\mu = \mathcal{H}^{n-1} \mres (B_1(0) \cap \{ x_n = 0 \})$ then 
\begin{equation}\label{eq:experiment}
    \int |\nabla F(0-y)| \mathrm{d}\mu(y) =  C_n \int_{ \{|y'| <1 \} }  \frac{1}{|y'|^{n-1}} \mathrm{d}y',
\end{equation}
where $y' = (y_1,...,y_{n-1})$ and $dy' = dy_1 ... dy_{n-1}$. This integral is infinite, but the order of the singularity is just critical for this to hold true. Since we integrate only the absolute value of $|\nabla F|$ in \eqref{eq:experiment}, we are not anymore able to take advantage of useful \emph{cancellation properties}. The idea of this cancellation is not new to mathematics and forms an integral part of Calderon-Zygmund theory. In this theory however, the symmetries and the structure of $\mathbb{R}^n$ are heavily used. Because of lacking symmetry of $\Gamma$, those arguments do not immediately carry over. 

Those considerations are obviously related to \emph{potential theory}, but our goal is to present an approach that does not rely on potential theory, except for standard elliptic regularity results. 

A second argument why $\mu = Q \; \mathcal{H}^{n-1} \mres \Gamma $  is a critical limit case is because of its criticality for the \emph{Wolff potential} 
\begin{equation}
    \mathcal{W}\mu(x) := \int_0^1 \frac{|\mu|(B_t(x))}{t^{n-1}} \; \frac{\mathrm{d}t}{t}.
\end{equation}
It is well known that $\mathcal{W}\mu(x) < \infty$ for all $x$  implies that solutions of \eqref{eq:measintro} lie in $W^{1,\infty}$, cf. \cite{Duzaar1}, \cite{Kilpelainen}. The fact that  $(\mathcal{H}^{n-1}\mres \Gamma)(B_t(x)) \sim t^{n-1} $ for suitably smooth  $\Gamma$ and all $x \in \Gamma$ leads once again to a critical singularity  in the definition of $W\mu$.

Equations like \eqref{eq:measspec} appear in applications.  Oftentimes they describe solutions of free boundary problems, for example the \emph{thin obstacle problem} (cf. \cite[Equation (3.8)]{FernandezReal}) and the \emph{Alt-Caffarelli Problem}  (cf. \cite[Equation(0.2)]{Alt}) as well as its biharmonic relative (cf. \cite[Theorem 1.4]{Marius}). We will explain possible applications in detail for the biharmonic Alt-Caffarelli problem, which has recently raised a lot of interest, cf. \cite{Dipierro2},\cite{Dipierro1},\cite{Marius}. Some methods, e.g the \emph{blow-up techniques} we use, are also inspired by  applications from free boundary problems \cite[Section 4]{Alt}, but need to be refined in our analysis.

The article is organized as follows: In Section 
\ref{sec:main} we introduce the underlying concept of weak solutions and state our main result. In Section \ref{sec:cont} we recall what regularity has been obtained in the existing literature and discuss in which sense $W^{1,\infty}$-regularity is the best regularity one may expect for our problem. Section \ref{sec:4} and 5 are devoted to the proofs of the main results. In Section \ref{sec:appli} we apply our results to study regularity of the biharmonic Alt-Caffarelli problem.



\section{Main Results}\label{sec:main}

Suppose for the rest of the article that $\Omega \subset \mathbb{R}^n$ is a bounded domain with smooth boundary and $\Gamma \subset \subset \Omega$ is such that $\mathcal{H}^{n-1}(\Gamma) < \infty$.
Further let $Q \in L^1( \Gamma)$, which is shorthand for $L^1( \Gamma, \mathcal{H}^{n-1} \mres \Gamma)$.

We next define a weak concept of solutions for  
\begin{equation}\label{eq:measdir}
\begin{cases} 
- \Delta v = Q \; \mathcal{H}^{n-1}\mres \Gamma  & \textrm{in} \; \Omega, \\ 
\quad \; \; v = 0 & \textrm{on} \; \partial \Omega.
\end{cases} 
\end{equation}
\begin{definition}[{cf. \cite[Definition 3.1]{Ponce}}]
We say that $v \in L^1(\Omega)$ is a weak solution of \eqref{eq:measdir} if for all $\phi \in C^2(\overline{\Omega})$ such that $\phi_{\mid_{\partial \Omega}} = 0$ one has 
\begin{equation}
- \int_\Omega v \; \Delta \phi \dx  =  \int_\Gamma Q \;  \phi \;  \mathrm{d}\mathcal{H}^{n-1}.
\end{equation}
\end{definition}

%



If $Q \in L^1( \Gamma, \mathcal{H}^{n-1})$ then $Q \mathcal{H}^{n-1}\mres \Gamma$ is a finite signed Radon measure and by
\cite[Proposition 3.2, Proposition 3.5]{Ponce} there exists a unique solution $v \in L^1(\Omega)$ of \eqref{eq:measdir} in the sense of the previous definition.  

As mentioned above, we are interested in the optimal regularity of such solutions, which we will identify as $W^{1,\infty}(\Omega)$. 

Let us first note that $W^{2,q}$-regularity is impossible for any solution of \eqref{eq:measdir} and any $q \in [1,\infty]$ unless $Q = 0$. This is due to the fact that $v \in W^{2,q}$ implies that $-\Delta v$ is absolutely continuous with respect to the Lebesgue measure, whereas the right hand side $Q \mathcal{H}^{n-1} \mres \Gamma$ is not absolutely continuous. 

This already yields that the best possible regularity in the sense of (integer) Sobolev spaces is given by $W^{1,\infty}(\Omega)$.

In classical function spaces however, the regularity could theoretically be improved. Recall that for smooth domains $\Omega \subset \mathbb{R}^n$ one has $W^{1,\infty}(\Omega) = C^{0,1}(\overline{\Omega})$, the space of Lipschitz functions.

 It turns out that no improvement of the $C^{0,1}$-regularity can be achieved in the classical function spaces  ---   $C^1$-regularity is impossible (cf. Section \ref{sec:Hopfolei}). 
 
The main theorems of this article state that Lipschitz regularity can actually be achieved under mild (and optimal) assumptions on $Q$ and $\Gamma$, which we shall discuss.



\begin{mainthm}[Lipschitz regularity for closed Lipschitz manifolds]\label{thm:lipreg}
Suppose that $\Omega \subset \mathbb{R}^n$ has smooth boundary and $\Omega' \subset \subset \Omega$ is  open with Lipschitz boundary  $\Gamma = \partial \Omega'$. Further suppose that $Q \in L^\infty(\Gamma)$. 
Then the unique solution to \eqref{eq:measdir} lies in $C^{0,1}(\overline{\Omega})$. 
\end{mainthm}

In this theorem, we have imposed additional restrictions on the set of admissible right hand sides, mainly that $\Gamma$ enjoys Lipschitz regularity and $Q \in L^\infty(\Gamma)$. Also the fact that we prescribe $\Gamma = \partial \Omega'$, i.e. $\Gamma$ is a \emph{closed} Lipschitz manifold, is an additional restriction of topological nature. 

We will see that $Q \in L^\infty(\Gamma)$ is a necessary condition for Lipschitz regularity, which justifies this additional restriction (cf. Lemma \ref{lem:neces}). 
The Lipschitz regularity of $\Gamma$ is necessary in the sense that it is not enough to demand that $\Gamma$ lies in $C^{0,\alpha}$ for any $\alpha \in (0,1)$ (cf. Remark \ref{rem:nohoelder}).

The topological requirement that $\Gamma = \partial \Omega'$ is used in our argumentation but by no means a necessary restriction. Indeed,  Main Theorem \ref{thm:lipreg} can be improved to hold with a weaker (and optimal) topological assumption.

\begin{mainthm}[Lipschitz regularity for compact Lipschitz manifolds]\label{thm:lipgraphmain}
Suppose that $\Omega \subset \mathbb{R}^n$ is a bounded domain with smooth boundary. Let $\Gamma$ be a compact Lipschitz manifold (with or without boundary). Then the unique solution of \eqref{eq:measdir} lies in $C^{0,1}(\overline{\Omega})$.
\end{mainthm}

Here the new assumption is compactness of $\Gamma$ which is indeed necessary (cf. Remark \ref{rem:nohoelder}). To clarify what we mean by ``compact Lipschitz manifold (with or without boundary)'' we refer to Appendix \ref{app:lipschitzDom}. From now on all compact manifolds are meant to be manifolds with or without boundary, unless explicitly stated otherwise.

The whole next section will be devoted to clarify the background of the Main Theorems, discussing existing literature, necessity of the assumptions and optimality.


\section{Context and Optimality} \label{sec:cont}

\subsection{The general measure valued problem}
In this section we recall some basic facts about the equation 
\begin{equation}\label{eq:measpoi}
\begin{cases} 
- \Delta v = \mu & in \; \Omega, \\ 
 \quad \; \; v = 0 & on \; \partial \Omega.
\end{cases} 
\end{equation}
where $\mu$ is a signed finite Borel measure on $\Omega$. We define the finite Borel measure $|\mu|$ as in \cite[pp. 137-139]{Rudin} and identify it with the  (outer) Radon measure 
\begin{equation}\label{eq:ximass}
\xi(A) := \inf_{U \supset A \; U \mathrm{Borel}} |\mu|(U) \quad ( A \subset X). 
\end{equation}
This way we can apply all results about outer measures also to $|\mu|$.

\begin{lemma}[{cf. \cite[Proposition 5.1]{Ponce}}]\label{lem:ponci}
Let $v$ be a weak solution of \eqref{eq:measpoi}. Then $v \in W_0^{1,q}(\Omega)$ for all  $ q \in [1, \frac{n}{n-1})$ and for all $\phi \in C_0^\infty(\Omega)$ one has 
\begin{equation}\label{eq:densi}
\int_\Omega \nabla v \nabla \phi  \; \mathrm{d}x = \int \phi \; \mathrm{d}\mu.
\end{equation} 
Moreover there exists $C= C(q)$ such that 
\begin{equation}
||v||_{W_0^{1,q}(\Omega)} \leq C |\mu|(\Omega). 
\end{equation}
and $\eqref{eq:densi}$ holds also for $\phi \in W_0^{1,q'}(\Omega)$, $(q ' >n, \frac{1}{q}+ \frac{1}{q'}= 1)$.
\end{lemma}

\subsection{Necessary criteria for Lipschitz continuity}\label{sec:neces}
In this article, our measure $\mu$ in \eqref{eq:measpoi} is always of the form $\mu = Q \mathcal{H}^{n-1} \mres \Gamma$ for some $Q \in L^\infty(\Gamma)$ and a hypersurface $\Gamma$. The reason for that is not just that this case is most relevant in the presented applications but also that this structure is in a way necessary for the Lipschitz continuity. The arguments presented here go back to observations in \cite{Kilpelainen}.   

\begin{lemma}[Necessity of the structure $\mu = Q \mathcal{H}^{n-1} \mres \Gamma$ and $Q \in L^\infty$]\label{lem:neces}
Suppose that $u \in C^{0,1}(\overline{\Omega})$ is a solution of \eqref{eq:measpoi} for a nonnegative measure $\mu$ supported on some Borel set $\Gamma$ such that $\mathcal{H}^{n-1}(\Gamma) <\infty$. Then there exists $Q \in L^\infty(\Gamma)$ such that $\mu = Q \mathcal{H}^{n-1} \mres \Gamma$. 
Moreover 
\begin{equation}\label{eq:estc}
||Q||_{L^\infty(\Gamma)} \leq 2^{2n-1} \alpha_n ||u||_{C^{0,1}(\overline{\Omega)}},
\end{equation}
where $\alpha_n = |B_1(0)|$ is the $n$-dimensional volume of a unit ball. 
\end{lemma} 


It is remarkable about the constant in \eqref{eq:estc} that it does not at all depend on $\Gamma$. 

\begin{proof}[Proof of Lemma \ref{lem:neces}]

Let $u \in C^{0,1}(\overline{\Omega})$ be as in the statement. Recall that $\mu = |\mu|$ and hence we can also identify $\mu$ with an outer Radon measure on $\mathbb{R}^n$. 
For $x_0 \in \Gamma$ and $r \in (0, \mathrm{dist}(x_0, \partial \Omega))$ set $\phi: \Omega \rightarrow \mathbb{R}$ defined by 
\begin{equation}
\phi(x) := \begin{cases} 
1 - \frac{|x-x_0|}{r}   & x \in B_r(x_0), \\ 0 & \textrm{otherwise}.
\end{cases}
\end{equation}
It is easy to check that then $\phi \in W_0^{1,q'}(\Omega)$ for some $q' > n$  and $\nabla \phi(x) = -\frac{x-x_0}{r|x-x_0|} \chi_{B_r(x_0)}(x)$ a.e..
 We infer that 
\begin{equation}
\int \phi \; \mathrm{d}\mu = \int_\Omega \nabla u \nabla \phi \; \mathrm{d}x .
\end{equation} 
Since $0 \leq \phi \leq 1$ and $\phi \geq \frac{1}{2}$ on $B_{\frac{r}{2}}(x_0)$ we conclude that 
\begin{align}
\frac{1}{2}\mu(B_{\frac{r}{2}}(x_0)) & \leq \left\vert \int \phi \; \mathrm{d}\mu \right\vert
 = \left\vert \int \nabla u \nabla \phi \dx \right\vert  \\ & \leq ||\nabla u ||_{L^\infty} \int_{B_r(x_0)} |\nabla \phi| \mathrm{d}x  = ||\nabla u||_{L^\infty} \frac{1}{r} |B_r(x_0)| 
 \leq ||u||_{C^{0,1}} \alpha_n r^{n-1}. \label{eq:aignulip} 
\end{align}
We infer that there exists $c=c(n,  ||u||_{C^{0,1}}) := 2^n \alpha_n ||u||_{C^{0,1}}$ such that for all $r < \frac{1}{2}\mathrm{dist}(\Gamma, \partial \Omega)$ and $x \in  \Gamma$ one has $\mu(B_r(x)) \leq cr^{n-1}$. If $x \not \in \Gamma$ then 
\begin{equation}
\mu(B_r(x)) = \mu(B_r(x) \cap \Gamma) \leq \begin{cases} \mu(B_{2r}(x_0)) & \textrm{if $\exists$  $x_0 \in B_r(x) \cap \Gamma$} \\ 0 &\textrm{otherwise} \end{cases}  \leq (2^{n-1}c) r^{n-1} .
\end{equation}
We conclude that $\mu(B_r(x)) \leq (2^{n-1}c) r^{n-1}$ for all $x \in \mathbb{R}^n$. Now \cite[Proposition 5.3]{Ponce} implies that for all $\delta \in (0,\frac{1}{2}\mathrm{dist}(x_0,\Gamma))$ and all Borel sets $A \subset \mathbb{R}^n$  one has $\mu(A) \leq 2^{n-1} c\mathcal{H}^{n-1}_\delta (A)$, where $\mathcal{H}^{n-1}_{\delta}$ denotes the $\delta$-Hausdorff capacity, cf. \cite[Definition B.1]{Ponce}. Letting $\delta \rightarrow 0$ we infer that $\mu(A) \leq 2^{n-1} c \mathcal{H}^{n-1}(A)$ for each Borel set $A \subset \mathbb{R}^n$. Since $\mu (A)= \mu(A \cap \Gamma)$ we infer that $\mu(A) \leq (2^{n-1} c) (\mathcal{H}^{n-1} \mres \Gamma )(A)$ for all Borel sets $A$.
Using \eqref{eq:ximass} and the fact that $\mathcal{H}^{n-1}\mres \Gamma$ is a Radon measure we infer that $\mu(A) \leq 2^{n-1} c (\mathcal{H}^{n-1} \mres \Gamma )(A)$ for all $A \subset \mathbb{R}^n$. Hence $\mu$ is absolutely continuous with respect to $\mathcal{H}^{n-1} \mres \Gamma$ and we infer from \cite[Section 1.6.2]{EvGar} that  
\begin{equation}
\mu(A) = \int_A (D_{\mathcal{H}^{n-1}\mres \Gamma } \mu ) \; \mathrm{d}\mathcal{H}^{n-1} \mres \Gamma,   \quad \forall A \subset \mathbb{R}^n \; \textrm{Borel}
\end{equation}
where $\mathcal{H}^{n-1} \mres \Gamma$ a.e. one can estimate
\begin{equation}
Q(x) := (D_{\mathcal{H}^{n-1}\mres \Gamma } \mu) (x)  \leq \limsup_{r \rightarrow 0+} \frac{\mu(B_r(x))}{\mathcal{H}^{n-1} \mres \Gamma (B_r(x))}  \leq 2^{n-1} c. 
\end{equation}
Hence $\mu = Q  \; \mathcal{H}^{n-1} \mres \Gamma$ for some $Q \in L^\infty(\Gamma)$. To finally prove the estimate \eqref{eq:estc} we conclude
\begin{equation}
||Q||_{L^\infty( \Gamma) } \leq 2^{n-1} c = 2^{2n-1} \alpha_n ||u||_{C^{0,1}}. \qedhere
\end{equation} 
\end{proof}

\begin{remark}[Necessity of Lipschitz regularity of $\Gamma$]\label{rem:nohoelder}
In this remark, we intend to look at the effect of the regularity of $\Gamma$ on the solution. To eliminate other influential factors we assume $Q = 1$.

 The main observation of this remark is as follows:
 requiring that $\Gamma$ is a $C^{0, \alpha}$-graph for some $\alpha  \in (0,1)$ is not enough to obtain the regularity. 
This becomes visible when looking at  \cite[Remark 2.7]{Kilpelainen2} or \eqref{eq:aignulip} of the present article. The computations there reveal the following: If $u \in C^{0,1}(\overline{\Omega})$ is a solution of \eqref{eq:measpoi} with a nonnegative measure $\mu$  on the right hand side then there exists $C> 0$ such that 
\begin{equation}\label{eq:muball}
    \mu(B_r(x)) \leq C r^{n-1} \quad \forall r > 0 \quad \forall x \in \Omega.
\end{equation}
Hence \eqref{eq:muball} forms a necessary criterion for Lipschitz regularity.
Next we give an example of a $C^{0,\alpha}$-graph $\Gamma$ that does not satisfy \eqref{eq:muball}. To this end we assume $n = 2$, $\Omega = B_2(0)$ and fix $\alpha \in [0,1)$. Choose $\alpha' \in (\alpha,1)$ and define 
\begin{equation}
    \Gamma := \left\lbrace \left(x ,  \frac{1}{1+\alpha'} x^{1+\alpha'} \sin \frac{1}{x} \right) : x \in [0,1] \right\rbrace.
\end{equation}
It is easy to show that $f(x):= \frac{1}{1+ \alpha'}x^{1+\alpha'} \sin \frac{1}{x}, (x \in [0,1])$, lies in $W^{1,p}(0,1)$ for all $p < \frac{1}{1-\alpha'}$, which embeds into $C^{0,\alpha}([0,1])$ if we choose $p < \frac{1}{1- \alpha'}$ suitably large. One readily checks that 
\begin{equation}
    \mathcal{H}^1(\Gamma) = \int_0^1 \sqrt{1 + f'(x)^2} \dx < \infty. 
\end{equation}
Using that $|f(x)|\leq |x|$ for all $x \in [0,1]$ and the mean value theorem we find
\begin{align}
    \mu(& B_r(0)) = \mathcal{H}^1(\Gamma \cap B_r(0)) = \int_{\{ x \in [0,1] : x^2+ f(x)^2 < r^2 \}} \sqrt{1+ f'(x)^2} \dx 
    \\ & \geq  \int_0^\frac{r}{\sqrt{2}} \sqrt{1  + \left(  x^{\alpha'} \sin \frac{1}{x} - \frac{1}{1+\alpha' }x^{\alpha' -1 } \cos \frac{1}{x} \right)^2} \dx 
    \\ &  \geq \int_0^\frac{r}{\sqrt{2}} \sqrt{1  - \left(x^{\alpha'} \sin \frac{1}{x} \right)^2 + \frac{1}{2} \left( \frac{1}{1+ \alpha'} x^{\alpha' - 1} \cos \frac{1}{x} \right)^2 } \dx 
    \\ & \geq \int_0^\frac{r}{\sqrt{2}} \frac{1}{\sqrt{2}(1+\alpha')} x^{\alpha'-1} \left\vert \cos{\frac{1}{x}} \right\vert \dx 
     \geq \frac{1}{\sqrt{2}(1+ \alpha')}\sum_{k > \frac{1}{2\pi} \left( \frac{\sqrt{2}}{r}+ \frac{\pi}{4}\right)}\int_{\frac{1}{2\pi k + \frac{\pi}{4}}}^\frac{1}{2\pi k - \frac{\pi}{4}}  x^{\alpha'-1} \left\vert \cos{\frac{1}{x}} \right\vert \dx 
    \\ & \geq  \frac{1}{2(1+\alpha')} \left( \sum_{k > \frac{1}{2\pi} \left( \frac{\sqrt{2}}{r}+ \frac{\pi}{4}\right)}\int_{\frac{1}{2\pi k + \frac{\pi}{4}}}^\frac{1}{2\pi k - \frac{\pi}{4}}  x^{\alpha'-1} \dx \right) 
    \\ &  = \frac{1}{2(1+\alpha')} \frac{1}{\alpha'} 
    \sum_{k > \frac{1}{2\pi} \left( \frac{\sqrt{2}}{r}+ \frac{\pi}{4}\right)}  \left( \frac{1}{(2 k \pi - \frac{\pi}{4})^{\alpha'}} - \frac{1}{(2 k \pi + \frac{\pi}{4})^{\alpha'}}\right) 
    \\ & \geq \frac{1}{2(1+\alpha')}  \sum_{k > \frac{1}{2\pi} \left( \frac{\sqrt{2}}{r}+ \frac{\pi}{4}\right)} \frac{\frac{\pi}{2}}{(2 k \pi + \frac{\pi}{4})^{\alpha' + 1}} \geq \frac{\pi}{4(1+\alpha')}  \sum_{k > \frac{1}{2\pi}  \left( \frac{\sqrt{2}}{r}+ \frac{\pi}{4}\right)} \int_{k}^{k+1} \frac{1}{(2 k \pi + \frac{\pi}{4})^{\alpha' + 1}} \dz
    \\ & \geq \frac{\pi}{4(1+\alpha')}  \sum_{k > \frac{1}{2\pi}  \left( \frac{\sqrt{2}}{r}+ \frac{\pi}{4}\right)} \int_{k}^{k+1} \frac{1}{(2 \pi z + \frac{\pi}{4})^{\alpha' + 1}} \dz
    \\ & \geq \frac{\pi}{4(1+\alpha')} \int_{\frac{1}{2\pi}  \left( \frac{\sqrt{2}}{r}+ \frac{\pi}{4}\right)}^\infty \frac{1}{(2\pi z + \frac{\pi}{4})^{\alpha' + 1}} \dz
    = \frac{\pi}{8(1+\alpha')} \left( \frac{\sqrt{2}}{r} +\frac{\pi}{2} \right)^{-\alpha'}. 
\end{align}
  Now if we assume that $\mu(B_r(0)) \leq Cr$ for all $r> 0$ we would obtain 
  \begin{equation}
      C r \geq \frac{1}{8(1+\alpha')} \left( \frac{\sqrt{2}}{r} +\frac{\pi}{2} \right)^{-\alpha'} .
  \end{equation}
Dividing by $r^{\alpha'}$ and then letting $r \rightarrow 0$ we obtain 
$
    0 \geq  \frac{1}{8(1+\alpha')} \sqrt{2}^{-\alpha'}
$, a contradiction.

As another conclusion from this example we can state that \emph{compactness} of $\Gamma$ is necessary. This is so since 
    \begin{equation}
    \Gamma_0 := \left\lbrace \left(x ,  \frac{1}{1+\alpha'} x^{1+\alpha'} \sin \frac{1}{x} \right) : x \in (0,1) \right\rbrace.
\end{equation}
is indeed a (noncompact) Lipschitz manifold, as graph of a function that is locally Lipschitz. Since $\Gamma_0$ coincides $\mathcal{H}^{n-1}$ a.e. with the counterexample above we infer that solutions to \eqref{eq:measdir} with $Q =1$ and $\Gamma = \Gamma_0$ are not $C^{0,1}(\overline{\Omega})$-regular.  
 \end{remark}

\subsection{Hölder regularity via duality method}\label{sec:dualmeth}

In this section we report on existing regularity statements briefly and discuss why they are of limited use when it comes to the optimal Lipschitz regularity. The famous duality method can be used to show  

\begin{prop}[$C^{0,\alpha}$-regularity]\label{prop:C0alpha}
Let $\Omega \subset \mathbb{R}^n$ and $\Gamma$ be a compact Lipschitz manifold. 
Let $v \in L^1(\Omega)$ be a weak solution of \eqref{eq:measdir}. Then $v\in W^{1,p}(\Omega)$ for all $p \in [1,\infty)$ and $v \in C^{0,\alpha}(\overline{\Omega})$ for all $\alpha \in [0,1)$. Furthermore, $\nabla v \in BMO(\Omega)$. 
\end{prop}

To prove this we convert \eqref{eq:measdir} into an equation of the form $-\Delta u = \mathrm{div}(F)$ for some $F \in L^\infty(\Omega,\mathbb{R}^n)$. For $p \in (1,\infty)$ it is classical (cf. \cite{Breit}) that $F \in L^p(\Omega,\mathbb{R}^n)$ implies that there exists a solution $u \in W_0^{1,p}(\Omega)$. This result does not generalize to $p = \infty$, as we shall see in the end of this section. 



\begin{lemma}\label{lem:divF} Let $\Gamma$ be a compact Lipschitz manifold.
Distributionally one has $Q \; \mathcal{H}^{n-1}\mres \Gamma = \mathrm{div}(F)$ for some $F \in L^\infty(\Omega; \mathbb{R}^n)$. Moreover $Q \; \mathcal{H}^{n-1} \mres \Gamma \in W_0^{1,1}(\Omega)^*$ and 
\begin{equation}\label{eq:estiF}
||F||_{L^\infty(\Omega)}=  || Q \; \mathcal{H}^{n-1} \mres \Gamma  ||_{W_0^{1,1}(\Omega)^*}  \leq C(\Omega) [\Gamma]_{Lip} ||Q||_{L^\infty(\Gamma)},  
\end{equation}
where $[\Gamma]_{Lip}$ is defined in Appendix \ref{app:lipschitzDom}.
\end{lemma} 
\begin{proof}
We first show that $Q \; \mathcal{H}^{n-1} \mres \Gamma \in (W_0^{1,1}(\Omega))^*$, i.e. for all $\phi \in C_0^\infty(\Omega)$ one has 
\begin{equation}\label{eq:Wo1prime}
\int_\Gamma |\phi| |Q| \; d \mathcal{H}^{n-1}  \leq D [\Gamma]_{Lip} ||Q||_{L^\infty} ||\nabla \phi||_{L^1(\Omega)},
\end{equation} 
for some $D= D( \Omega)$.
By \cite[Proposition 17.17]{Ponce} it suffices to show that for all Borel sets $A \subset \mathbb{R}^n$,
\begin{equation}
\int_A |Q| \; \mathrm{d} \mathcal{H}^{n-1} \mres \Gamma \leq \widetilde{D}[\Gamma]_{Lip} ||Q||_{L^\infty} \mathcal{H}^{n-1}_\infty (A) ,  
\end{equation}
 where $\mathcal{H}^{n-1}_\infty$ is the $\infty$-Hausdorff capacity (cf. \cite[Definition B.1]{Ponce}) and $\widetilde{D} = \widetilde{D}(  \Omega)$. Using \cite[Proposition B.3]{Ponce} we find that it suffices to show that for all $y \in \mathbb{R}^n$ and $r > 0$ one has 
\begin{equation}
\int_{B_r(y)\cap \Gamma} |Q| d\mathcal{H}^{n-1} \leq  \widetilde{D}[\Gamma]_{Lip} ||Q||_{L^\infty}  \alpha_{n-1} r^{n-1}.
\end{equation}
To this end cover $\Gamma$ with open sets $R_1(U_1\times V_1) ,...,R_m(U_m \times V_m) \subset \mathbb{R}^n$ such that $U_i \subset \mathbb{R}^{n-1}$ open and $V_i \subset \mathbb{R}$ open and  $R_i \in O(n)$ such that for $ f_i \in W^{1,\infty}(U_i)$ one has $\Gamma \cap R_i( U_i \times V_i)  = R_i\{ (x',f_i(x')) : x' \in V_i \}$  . Then we can estimate 
\begin{align}
\int_{B_r(y)\cap \Gamma} |Q| d\mathcal{H}^{n-1} & \leq ||Q||_{L^\infty} \mathcal{H}^{n-1} ( B_r(y) \cap \Gamma) \leq ||Q||_{L^\infty}  \sum_{i = 1}^m  \mathcal{H}^{n-1} ( B_r(y) \cap \Gamma \cap R_i(U_i \times V_i))
\\ &  = ||Q||_{L^\infty} \sum_{i = 1}^{m}  \int_{ \{x' \in U_i : (x', f_i(x')) \in B_r(R_i^{-1}(y)) \} } \sqrt{1 + |\nabla f_i(x')|^2} \; \mathrm{d}x'.
\end{align}
Now note that $(x',f_i(x')) \in B_r(R_i^{-1}(y))$ implies that $x' \in B_r ( y_i' )$, where $y_i' := (R_i^{-1}(y)^{(1)},...,R_i^{-1}(y)^{(n-1)})^T \in \mathbb{R}^{n-1}$. Thus 
\begin{align}
\int_{B_r(y)\cap \Gamma} |Q| d\mathcal{H}^{n-1} & \leq ||Q||_{L^\infty} \sum_{i = 1}^m \sqrt{1 + ||\nabla f_i||_{L^\infty(V_i)}^2} \int_{B_r(y_i)} \; \mathrm{d}x'  \\ & \leq  ||Q||_{L^\infty}  \left( \sum_{i = 1}^m \sqrt{1 + ||\nabla f_i||_{L^\infty(V_i)}^2} \right) \alpha_{n- 1} r^{n-1}   .
\end{align}
We infer from this and \eqref{eq:GammaLipschitzNorm} that 
\begin{equation}
  \int_{B_r(y)\cap \Gamma} |Q| d\mathcal{H}^{n-1}  \leq ||Q||_{L^\infty} [\Gamma]_{Lip}\alpha_{n-1}r^{n-1}.  
\end{equation}
Equation \eqref{eq:Wo1prime} and hence the second estimate in \eqref{eq:estiF} follows by \cite[Proposition B.3]{Ponce} and \cite[Proposition 17.17]{Ponce}. The existence of $F \in L^\infty( \Omega; \mathbb{R}^n)$ follows directly from \cite[Lemma 6.6]{Torres}. 
\end{proof}


\begin{lemma}\label{lem:310} Let $\Gamma$ be a compact Lipschitz manifold and let $v $ be a weak solution of \eqref{eq:measdir}. Then $v \in W_0^{1,2}(\Omega) \cap L^\infty(\Omega)$. Furthermore there exists a constant $D= D(\Omega) > 0$ such that
\begin{equation}
    ||v||_{L^\infty(\Omega)} \leq D(\Omega) || Q  \; \mathcal{H}^{n-1} \mres \Gamma||_{W_0^{1,1}(\Omega)^*}.
\end{equation}
\end{lemma}
\begin{proof}
We first show that $v \in W_0^{1,2}(\Omega)$.
 Recall that $v \in W_0^{1,q}(\Omega)$ for some $q > 1$ by Lemma \ref{lem:ponci}. 
Now fix $\phi \in C_0^\infty( \Omega )$. We obtain by Lemma \ref{lem:divF} for $C_1 := ||Q \; \mathcal{H}^{n-1} \mres \Gamma ||_{W_0^{1,1}(\Omega)^*}$ and $C_2 := \sqrt{|\Omega|} C_1$ that
\begin{equation}
\left\vert \int_\Omega \nabla v \nabla \phi \dx \right\vert =  \left\vert \int \phi Q \; \mathrm{d}\mathcal{H}^{n-1} \right\vert \leq C_1||\nabla \phi||_{L^1(\Omega)}   \leq C_2||\nabla \phi||_{L^2(\Omega)}  = C_2 ||\phi||_{W_0^{1,2}(\Omega)}.
\end{equation}
Since $W_0^{1,2}(\Omega)$ is a Hilbert space and $C_0^\infty(\Omega)$ is dense in $W_0^{1,2}(\Omega)$ there exists $w \in W_0^{1,2}(\Omega)$ such that 
\begin{equation}
\int_\Omega \nabla v \nabla \phi \; \mathrm{d}x = \int \nabla w \nabla \phi \;  \mathrm{d}x \quad \forall \phi \in C_0^\infty(\Omega).
\end{equation}
We infer that $u:= v-w \in W_0^{1,q}(\Omega)$ is a weak solution of 
\begin{equation}
\begin{cases}
-\Delta u = 0 & \mathrm{in } \; \Omega, \\ \quad \; \;  u = 0 & \mathrm{on }  \; \partial \Omega. 
\end{cases}
\end{equation}
By uniqueness of such solution in $W_0^{1,q}$ (cf. \cite[Section 2.5.2]{Sweers}) we infer that $v = w \in W_0^{1,2}(\Omega)$. It remains to show the $L^\infty$-regularity and the desired estimate. By Lemma \ref{lem:divF} we infer that $v \in W_0^{1,2}(\Omega)$ is a weak solution of $-\Delta v = \mathrm{div}(F)$ for some $F \in L^\infty(\Omega; \mathbb{R}^n)$. The claim follows then immediately from \cite[Lemma 5.2]{Ponce}. 
\end{proof}
Note that the previous regularity is definitely not true for solutions of \eqref{eq:measpoi} with arbitrary measures $\mu$. Indeed, if $\mu= \delta_{x_0}$ for some $x_0 \in \Omega$ then the solution of \eqref{eq:measpoi} is Green's function  $G_\Omega(x_0, \cdot)$, which does not lie in $L^\infty(\Omega)$, cf. \cite[Section 4]{Sweers}.

\begin{proof}[Proof of Proposition \ref{prop:C0alpha}]
We infer from the previous lemma and Lemma \ref{lem:divF} that each solution $v$ of \eqref{eq:measdir} lies in $W_0^{1,2}(\Omega)$ and is a weak solution of $- \Delta v = \mathrm{div}(F)$ for some $F \in L^\infty(\Omega)$ . In particular $F \in BMO(\Omega)$. Applying \cite[Theorem 2.9]{Breit} with $p=2$ and $\omega \equiv 1$ we infer that $\nabla v \in BMO(\Omega)$. Since $BMO(\Omega) \subset L^p(\Omega)$ for all $p > 1$, we infer that $v \in W^{1,p}(\Omega)$ for all $p \in (1,\infty)$. By Sobolev embedding $v \in C^{0,\alpha}(\overline{\Omega})$ for all $\alpha \in [0,1)$.
\end{proof}
Finally we discuss the limitations of the duality method to $p = \infty.$

\begin{remark}[Limitations of the duality method]
  To this end we construct  $F \in L^\infty(\Omega;\mathbb{R}^n)$ such that $-\Delta u = \mathrm{div}(F)$ has no distributional solution in $W^{1,\infty}(\Omega)$.

The argument we present is a slight variation of arguments in \cite[Section 2]{Russ}. 
Suppose that $\Omega \subset \mathbb{R}^2$ is a smooth domain and $\Omega \supset [0,1]^2$. 
By \cite[Proof of Proposition 2, Section 2]{Russ} there exists a sequence $(\psi_j)_{j = 1}^\infty \subset C_0^\infty(\Omega)$ such that 
\begin{equation}
    \left\Vert \partial^2_{x_1x_1}\psi_j \right\Vert_{L^1(\Omega)} + \left\Vert \partial^2_{x_2x_2} \psi_j \right\Vert_{L^1(\Omega)} \leq 1 \quad \forall j \in \mathbb{N}, \quad  \quad \quad  \left\Vert \partial^2_{x_1x_2} \psi_j \right\Vert_{L^1(\Omega)} \geq j \quad \forall j \in \mathbb{N}.
\end{equation}
Now define 
\begin{equation}
    T_j : L^\infty(\Omega) \rightarrow \mathbb{R}, \quad T_j(g) := \int_\Omega g \; \partial^2_{x_1x_2} \psi_j \; \mathrm{d}x.
\end{equation}
One readily checks that $||T_j||_{L^\infty(\Omega)^*} = \left\Vert \partial^2_{x_1x_2} \psi_j \right\Vert_{L^1(\Omega)}  \rightarrow \infty$ as $j \rightarrow \infty$. 
By the Banach-Steinhaus theorem 
there exists $G \in L^\infty(\Omega)$ such that (along a subsequence) 
\begin{equation}\label{eq:GLIM}
   \left\vert T_j(G) \right\vert =  \left\vert \int_\Omega G \;  \partial^2_{x_1x_2} \psi_j \dx \right\vert \rightarrow \infty  \quad (j\rightarrow \infty).
\end{equation}
Define now $F(x) := (0, G(x))^T$ for almost every $x \in \Omega$. Clearly $F \in L^\infty(\Omega;\mathbb{R}^2)$. We claim that $-\Delta u = \mathrm{div}(F)$ has no distributional solution $u \in W^{1,\infty}(\Omega)$. For a contradiction we assume the opposite. Then there exists $u \in W^{1,\infty}(\Omega)$ such that for all $\phi \in C_0^\infty(\Omega)$ one has  
\begin{equation}
    \int_\Omega \nabla u \cdot \nabla \phi \dx  = \int_\Omega F \cdot  \nabla \phi \dx. 
\end{equation}
Plugging in $\phi= \partial_{x_1}\psi_j$  and integrating by parts we find 
\begin{equation}
    \int_\Omega G \; \partial^2_{x_1x_2} \psi_j \dx = \int_\Omega ( \partial_{x_1} u \partial^2_{x_1x_1} \psi_j + \partial_{x_2} u \partial^2_{x_1x_2} \psi_j ) \dx  = \int_\Omega \partial_{x_1} u ( \partial^2_{x_1x_1} \psi_j + \partial^2_{x_2x_2} \psi_j ) \dx . 
\end{equation}
Taking absolute values and using the estimates above we find
\begin{equation}
    \left\vert \int_\Omega G  \; \partial^2_{x_1x_2} \psi_j \dx \right\vert  \leq ||\partial_{x_1} u||_{L^\infty} \leq ||u||_{W^{1,\infty}}.
\end{equation}
Letting $j \rightarrow \infty$ we obtain a contradiction to \eqref{eq:GLIM}.
\end{remark}

\subsection{The Hopf-Oleinik-Lemma and $C^1$-regularity}\label{sec:Hopfolei}
In this section we explain why one can not hope for more than $C^{0,1}$-regularity in classical function spaces. More precisely we show 

\begin{prop}[Impossibility of $C^1$-regularity] \label{prop:Hopfgonebad}
Suppose that $\Omega \subset \mathbb{R}^n$ is a domain and let $\Gamma = \partial \Omega'$ for some domain $\Omega' \subset \subset \Omega$ with $C^{1,\beta}$-boundary for some $\beta \in (0,1)$. Let $v \in L^1(\Omega)$ be a weak solution of \eqref{eq:measdir}. Then $ v \not \in C^1(\Omega)$ unless $Q = 0$. 
\end{prop}

The main ingredient used in the proof is the famous Hopf-Oleinik boundary point lemma stated here (in a nonoptimal version).

\begin{lemma}[{Hopf-Oleinik-Lemma, cf. \cite[Theorem 4.1]{Lieberman}}] \label{lem:Hopf}
Suppose that $D \subset \mathbb{R}^n$ is a set with $C^{1,\beta}$-boundary for some $\beta > 0$. 
Let $v \in C^2(D) \cap C(\overline{D}) $ be a harmonic function in $D$. If $x_0 \in \partial D$ is such that $v(x_0) = \max_{x \in \partial D} v(x)$ then one has
\begin{equation}
\liminf_{t \rightarrow 0- } \frac{v(x_0+ t \nu(x_0)) - v(x_0)}{t} > 0, 
\end{equation}
where $\nu(x_0)$ denotes the outer unit normal at $x_0$. If $x_0 \in \partial D$ is such that $v(x_0) = \min_{x \in \partial D} v(x)$ then one has 
\begin{equation}
\limsup_{t \rightarrow 0- } \frac{v(x_0+ t \nu(x_0)) - v(x_0)}{t} < 0.
\end{equation}
\end{lemma}

\begin{proof}[Proof of Proposition \ref{prop:Hopfgonebad}]
Assume that $\Gamma, \Omega'$ are as in the statement and let $Q \in L^\infty(\Gamma) \setminus \{ 0 \}$.  Moreover let $v \in L^1(\Omega)$ be as in the statement.  Assume now that $v \in C^1(\Omega)$. Since $v$ is harmonic in a neighborhood of $\partial \Omega$ and attains homogeneous boundary values on $\partial \Omega$  one infers by Schauder theory that $v \in C^1(\overline{\Omega})$. In particular $\max_{x \in \overline{\Omega}} v(x), \min_{x \in \overline{\Omega}} v(x)$ are attained. Let $x_0 ,x_1 \in \overline{\Omega}$ be such that $v(x_0) = \max_{x \in \overline{\Omega}} v(x)$ and $v(x_1) = \min_{x \in \overline{\Omega}} v(x)$. We claim next that either $x_0$ or $x_1$  can be chosen to lie on $\Gamma$. Indeed, one has by harmonicity of $v$ on $\Omega'$ 
\begin{equation}\label{eq:v1}
\max_{x \in \overline{\Omega'}} v(x) = \max_{x \in \Gamma} v(x)  
\end{equation} 
and by harmonicity of $v$ on $\Omega'' := \Omega \setminus \overline{\Omega'}$ 
\begin{equation}\label{eq:v2}
\max_{x \in \overline{\Omega''}} v(x)= \max_{x \in \partial \Omega''} v(x) = \max_{x \in \Gamma \cup \partial \Omega} = \max \{ 0, \max_{x \in \Gamma} v(x) \},
\end{equation}
as $v = 0 $ on $\partial \Omega$. If now $\max_{x \in \overline{\Omega}} v(x) > 0$ then \eqref{eq:v1} and \eqref{eq:v2} leave the conclusion that
\begin{equation}
\max_{x \in \overline{\Omega'}} v(x) = \max_{x \in \overline{\Omega''}} v(x) = \max_{x \in \Gamma} v(x).
\end{equation}
Since $\overline{\Omega'} \cup \overline{\Omega''} = \overline{\Omega}$ we obtain that $x_0$ can be chosen to lie on $\Gamma$. On contrary if $\max_{x \in \overline{\Omega}} v(x) \leq 0$ then $\min_{x \in \overline{\Omega}} v(x) < 0 $ unless $v \equiv 0$ (which is not allowed as $Q \equiv 0$ was excluded). Now we can argue the same way as in \eqref{eq:v1}, \eqref{eq:v2} to find  
\begin{equation}
\min_{x \in \overline{\Omega'}} v(x) = \min_{x \in \Gamma} v(x) , 
\end{equation}  
\begin{equation}
\min_{x \in \overline{\Omega''}} v(x)= \min_{x \in \partial \Omega''} v(x) = \min \{ 0, \min_{x \in \Gamma} v(x) \}.
\end{equation}
Since now $\min_{x \in \overline{\Omega}} v(x) < 0$ we infer that 
\begin{equation}
\min_{x \in \overline{\Omega'}} v(x) = \min_{x \in \overline{\Omega''}} v(x) = \min_{x \in \Gamma} v(x)
\end{equation}
and since $\overline{\Omega'} \cup \overline{\Omega''} = \overline{\Omega}$ we once more infer that $x_1$ can be chosen to lie on $\Gamma$. Without loss of generality we can assume that $x_0$ lies on $\Gamma$, otherwise we may consider $-v$. Now we apply Lemma \ref{lem:Hopf} $D := \Omega'$. Since $\max_{x \in D} v(x)$ is attained at $x_0$ and $v \in C^1(\overline{D})$ we obtain that 
\begin{equation}
0 < \liminf_{t \rightarrow 0-} \frac{v(x_0+ t \nu_{\Omega'}(x_0)) - v(x_0)}{t} = \nabla v(x_0) \cdot \nu_{\Omega'}(x_0). \label{eq:w1}
\end{equation}
Analogously we can also consider $D= \Omega''$  (or one connected component of $\Omega''$ whose boundary  contains $x_0$) and apply Lemma \ref{lem:Hopf}. Since $\max_{x \in D} v(x)$ is also attained at $x_0$ we obtain 
\begin{equation}\label{eq:w2}
0 < \liminf_{t \rightarrow 0-} \frac{v(x_0+ t \nu_{\Omega''}(x_0)) - v(x_0)}{t}  = \nabla v(x_0) \cdot \nu_{\Omega''}(x_0).
\end{equation}
But now $\nu_{\Omega''} (x_0) = - \nu_{\Omega'} (x_0)$ as $x_0 \in \Gamma$ and $\Omega'$ and $\Omega''$ lie on two different sides of $\Gamma$. This means however that  \eqref{eq:w1} and \eqref{eq:w2} cannot hold at the same time. A contradiction. 
\end{proof}

\subsection{Remarks on potential theory} \label{sec:35}

As already announced our approach will not use potential theory, except for standard regularity results. Nevertheless we comment briefly on  potential theoretic results about \eqref{eq:measdir}, since one can indeed give a partial positive answer to our question of $W^{1,\infty}$ regularity. 

As in the introduction we consider for $Q \in L^1(\Gamma)$ the \emph{single-layer potential}
\begin{equation}
 \mathcal{S}Q(x) := \int_\Gamma F(x-y)  Q(y) \; \mathrm{d}\mathcal{H}^{n-1}(y),
\end{equation}
where $F$ is the fundamental solution of $-\Delta$ in $\mathbb{R}^n$. 
We intend to understand the behavior of $\mathcal{S}Q$ and its partial derivatives on $\Gamma$. Note first that $\mathcal{S}Q $ is smooth and harmonic on $\mathbb{R}^n\setminus \Gamma$. 
Further one can show that $\mathcal{S}Q \in W^{1,q}_{loc}(\mathbb{R}^n)$ for some $q \in [1, \frac{n}{n-1})$ and in the sense of weak derivatives one has 
\begin{equation}
    \nabla (\mathcal{S}Q) (x) = \int_\Gamma \nabla F(x-y) Q(y) \; \mathrm{d}\mathcal{H}^{n-1}  \quad a.e. \; x \in \mathbb{R}^n. 
\end{equation}

A first remarkable theorem discusses a control of the normal derivatives of $\mathcal{S}Q$ on $\Gamma$ under certain assumptions on $\Gamma,Q$. 

\begin{prop}[{Normal derivative of single-layer potential, \cite[Theorem 14.IV]{Miranda}}]\label{prop:Mira1} Suppose that $\Gamma= \partial T \subset \subset  \mathbb{R}^n$ for some $C^{1,\lambda}$-domain $T\subset \mathbb{R}^n$, $\lambda > 0 $ with unit normal field $\nu$. 
If $Q \in L^1(\Gamma)$ then the inner and outer normal derivatives $\partial_\nu^+ \mathcal{S}Q$ and $\partial_\nu^- \mathcal{S}Q$ exist  in the sense of 
\begin{equation}\label{eq:derpot}
    \partial_\nu^{\pm} \mathcal{S}Q = \lim_{ t \rightarrow 0 \pm } \nabla \mathcal{S}Q(x_0 + t \nu(x_0)) \cdot \nu(x_0) \quad  \; \textrm{exists for} \;  \mathcal{H}^{n-1} a.e. \; x_0 \in \Gamma 
\end{equation}
and 
\begin{equation}
    \partial_\nu^{\pm}(\mathcal{S}Q)(x_0) = \mp \frac{1}{2}Q(x_0) + \int_\Gamma \partial_{\nu} F(x_0 - y) Q(y) \; \mathrm{d}\mathcal{H}^{n-1}(y)  \quad \quad \mathcal{H}^{n-1} a.e. \;  x_0 \in \Gamma. 
\end{equation}
If $Q \in L^p(\Gamma)$ for some $p > \frac{n-1}{\lambda}$ then the map $\Gamma \ni x_0 \mapsto \int_{\Gamma} \partial_\nu F(x_0 - y) Q(y) \; \mathrm{d}\mathcal{H}^{n-1}(y)$ lies in $C^{0,\mu}(\Gamma)$ for any $\mu < \lambda - \frac{n-1}{p}$. In particular $\partial_\nu^{\pm} \mathcal{S} Q \in L^\infty(\Gamma)$ if  $Q \in L^\infty(\Gamma)$. If $Q \in C^0(\Gamma)$ then the above equations hold everywhere and not just $\mathcal{H}^{n-1}$ a.e..
\end{prop}
 This shows that $Q \in L^\infty(\Gamma)$ implies bounded normal derivatives on $\Gamma$ in the sense of \eqref{eq:derpot}.

 
 If $Q$ is somewhat more regular we can also bound non-normal derivatives on $\Gamma$ and hence the whole gradient. Indeed, one can show 
\begin{prop}[{\cite[Theorem 14.VII]{Miranda}}]
If $T \subset \subset \Omega$ is a $C^{1,\lambda}$-domain and $Q \in C^{0, \lambda}$ then $ \nabla \mathcal{S}Q \in C^{0,\lambda}(\overline{T}) \cap C^{0,\lambda} ( \overline{\Omega \setminus T})$. In particular $\mathcal{S}Q \in W^{1,\infty}(\Omega)$.
\end{prop}
While this is a positive result on our original question it is difficult to weaken the assumptions with potential theoretic methods. The reason for that is that layer potentials do not necessarily behave well with approximations and less regular domains. Another reason is that the sense of \eqref{eq:derpot} is a very weak sense of assuming boundary values and not compatible with modern notions of (BV/Sobolev)-traces. 

There are potential theoretic methods that allow for better control of boundary values, a common notion being the \emph{nontangential maximal function}. Using these methods one can indeed study the behavior of $\mathcal{S}Q$ on Lipschitz domains via approximation (cf. eg. \cite{Verchota}). The price one usually pays for this approximation is however the restriction to $W^{1,p}$ for $p \in (1,\infty)$. The reason for that is that crucial a-priori estimates like \cite[Lemma 1.3]{Verchota} have no trivial generalization to $p = \infty$.

 An observation which is highly relevant for the article is the \emph{normal jump} of size $Q$ on the boundary $\Gamma = \partial T$, i.e.
\begin{equation}\label{eq:normaljump}
    \partial_\nu^+ \mathcal{S}Q - \partial_\nu^- \mathcal{S}Q = - Q  \quad \textrm{on} \; \Gamma. 
\end{equation}
We will obtain this jump later independently with modern techniques from PDE and geometric measure theory.

\section{Proof of Main Theorem \ref{thm:lipreg}}\label{sec:4}

This section is devoted to the proof of our first main theorem. We will first obtain the desired Lipschitz regularity for smooth data $\Gamma$ and $Q$ by constructing \emph{almost-solutions} with the aid of the comparison function $x \mapsto \mathrm{dist}(x, \Gamma)$. Once Lipschitz regularity of solutions with smooth data is shown, we can obtain \emph{a-priori estimates} of the Lipschitz norm in terms of $\Gamma$ and $Q$. To find sharp a-priori estimates we perform \emph{blow-up procedures} and study the behavior of solutions around points in $\Gamma$. Finally we can argue by approximation to show the claim also for less regular data.  
\subsection{Regularity for smooth initial data}

First we show the result for surfaces of the form $\Gamma = \partial \Omega'$, where $\Omega'\subset \subset  \Omega$ is an open, bounded set with $C^2$-boundary.  The advantage of this additional smoothness is that one can 
work with the  \emph{signed distance function} $d_{\Omega'}$, see Appendix \ref{app:sign} for the precise definition and basic properties.
We will observe that the absolute value of the distance function $|d_{\Omega'}|$ already solves a similar problem than the one we intend to solve.
Thankfully $|d_{\Omega'}| \in W^{1,\infty}(\Omega)$, which is an important step towards our desired regularity. 

If we write in the sequel $\Gamma = \partial \Omega' \in C^k$ we mean that $\Omega' \subset \subset \Omega$ is a domain with $C^k$-smooth boundary $\Gamma$.  
The notation $Q \in W^{2,p} ( \Omega)$ shall indicate that $Q \in L^1(\Gamma)$ can be extended to a $W^{2,p}$-function on $\Omega$.

\begin{lemma}\label{lem:regudist}
Suppose that $Q \in W^{2,p}(\Omega)$, $n < p \leq \infty$ and $\Gamma = \partial \Omega' \in C^2$. Let $\epsilon > 0$ be such that $\Gamma_\epsilon := \{ x \in \Omega : \mathrm{dist}(x, \Gamma) < \epsilon \}$ is a $C^2$-domain and $d = d_{\Omega'} \in C^2(\Gamma_\epsilon)$. Then $|d| \in W^{1,2}(\Gamma_\epsilon)$ and for all $\phi \in C_0^\infty( \Gamma_\epsilon)$ one has 
\begin{align}\label{eq:disttest}
\int_{\Gamma_\epsilon} \nabla  \left( \frac{1}{2} Q |d| \right) \nabla \phi \dx & = -\int_\Gamma \phi Q d\mathcal{H}^1 + \int_{\Gamma_\epsilon} \left( \Delta \frac{Q}{2} d \right) \chi  \phi \; \mathrm{d}x ,
\end{align}
where $\chi = \mathbf{1}_{\Omega'} - \mathbf{1}_{\Omega \setminus \Omega'}$. 
\end{lemma}
\begin{proof}
That $|d| \in W^{1,2}(\Gamma_\epsilon)$ follows from $d \in W^{1,2}(\Gamma_\epsilon)$ and \cite[Theorem 4.4]{EvGar}. Next we divide $\Gamma_\epsilon = \Gamma_\epsilon^+ \cup \Gamma_{\epsilon}^- \cup \Gamma$, where $\Gamma_\epsilon^+:= \Gamma_\epsilon \cap \Omega \setminus \overline{\Omega'} $ and $\Gamma_\epsilon^{-} := \Gamma_\epsilon  \cap \Omega'$. Note that $d> 0$ on $\Gamma_\epsilon^+$ and $d< 0$ on $\Gamma_\epsilon^-$ and thus for $\phi \in C_0^\infty(\Gamma_\epsilon)$ one has  
\begin{align}
\int_{\Gamma_\epsilon} \nabla \left( \frac{1}{2} Q |d| \right) \nabla \phi \dx  & = \int_{\Gamma_{\epsilon}^+}   \nabla \left( \frac{1}{2} Q d \right) \nabla \phi  \dx - \int_{\Gamma_{\epsilon}^-}   \nabla \left( \frac{1}{2} Q d \right) \nabla \phi  \dx 
\\ & = \int_{\partial \Gamma_\epsilon^+}  \phi  \nabla \left( \frac{1}{2} Q d \right)  \nu_{\Gamma_\epsilon^+} \; \mathrm{d} \mathcal{H}^{n-1}  - \int_{\partial \Gamma_\epsilon^+} \Delta \left( \frac{Q}{2} d  \right) \phi \dx 
\\ \quad &  - \int_{\partial \Gamma_\epsilon^-} \phi \nabla \left( \frac{1}{2} Q d \right)  \nu_{\Gamma_\epsilon^-} \; \mathrm{d} \mathcal{H}^{n-1}  + \int_{\partial \Gamma_\epsilon^-} \Delta \left( \frac{Q}{2} d  \right) \phi \dx 
\\ & = \int_\Gamma \phi \nabla \left( \frac{1}{2} Q d \right)  ( \nu_{\Gamma_\epsilon^+} - \nu_{\Gamma_{\epsilon}^-} ) \; \mathrm{d}\mathcal{H}^{n-1}  + \int_{\Gamma_\epsilon} \Delta \left( \frac{Q}{2} d  \right)  \chi \phi \dx . 
\end{align}
Notice that on $\Gamma$ one has $\nu_{\Gamma_\epsilon^+} = - \nu_{\Gamma_\epsilon^-} = -\nu_{\Omega'}$ and hence 
\begin{equation}
\int_{\Gamma_\epsilon} \nabla \left( \frac{1}{2} Q |d| \right) \nabla \phi \dx  = -\int_\Gamma \phi \nabla (Qd) \nu_{\Omega'} \; \mathrm{d}\mathcal{H}^{n-1} + \int_{\Gamma_\epsilon} \Delta \left( \frac{Q}{2} d  \right)  \chi \phi \dx.
\end{equation}
Now on $\Gamma$ one has $d = d_{\Omega'} =0 $ and thus by Lemma \ref{lem:signdist}
$\nabla (Qd) = (\nabla Q) d + Q \nabla d = Q \nu_{\Omega'}.$ Here we have used that $Q \in C^1(\overline{\Omega})$ by Sobolev embedding. We infer 
\begin{equation}
\int_{\Gamma_\epsilon} \nabla \left( \frac{1}{2} Q |d| \right) \nabla \phi \dx   = - \int_\Gamma \phi Q \; \mathrm{d}\mathcal{H}^{n-1} + \int_{\Gamma_\epsilon} \Delta \left( \frac{Q}{2} d  \right)  \chi \phi \dx.
\end{equation} 
\end{proof}

\begin{cor}
Suppose that $Q \in W^{2,p}(\Omega)$, $n < p \leq \infty$ and $\Gamma = \partial \Omega' \in C^2$. Let $\epsilon > 0$ be such that for all $\epsilon' \leq \epsilon$ the set $\Gamma_{\epsilon'} := \{ x \in \Omega : \mathrm{dist}(x, \Gamma) < \epsilon' \}$ is a $C^2$-domain and $d = d_{\Omega'} \in C^2(\Gamma_\epsilon)$. Then the solution $v$ of \eqref{eq:measdir} satisfies $v \in C^{0,1}(\overline{\Omega})$.  
\end{cor}
\begin{proof}
We first show that $v \in C^{0,1}(\overline{\Gamma_{{\epsilon}/{2}}})$. To this end we look at $w := v + \frac{1}{2} Q  |d|$ which lies in $W^{1,2}(\Gamma_\epsilon)$ by Proposition \ref{prop:C0alpha} and satisfies for each $\phi \in C_0^\infty(\Gamma_\epsilon)$ by \eqref{eq:disttest}
\begin{equation}
\int_{\Gamma_\epsilon} \nabla w \nabla \phi = \int_{\Gamma_\epsilon} \left( \Delta \frac{Q}{2} d \right) \chi \phi \dx .
\end{equation}
Now note that $\Delta \left( \frac{Q}{2} d\right) \chi = \frac{\chi}{2}( d \Delta Q  + 2 \nabla Q \nabla d + Q \Delta d)$, which lies in $L^p(\Gamma_\epsilon)$.
We infer that $w \in W^{2,p}_{loc} ( \Gamma_\epsilon)$ for some $p \in (n, \infty)$ and hence $w \in C^1(\Gamma_\epsilon)$. Now $v = w - \frac{1}{2}Q |d| $ lies in $C^{0,1}$ since $|d(x)| = \mathrm{dist}(x,\Omega') + \mathrm{dist}(x, \Omega'^C)$ is Lipschitz continuous as sum of Lipschitz continuous functions and $Q, w \in C^1(\Gamma_\epsilon) \subset C^{0,1}(\overline{\Gamma_{\epsilon/2}})$.  The claimed Lipschitz continuity on $\overline{\Gamma_{\epsilon/2}}$ is shown. 
Since $v$ is harmonic on the $C^2$-domain $\Omega\setminus \Gamma_{\epsilon/4} $ and $v\vert_{\partial ( \Omega \setminus \Gamma_{\epsilon/4})}$ is smooth we infer by elliptic regularity that $v \in W^{2,n+1}(\Omega \setminus \Gamma_{\epsilon/4}) \subset C^1(\overline{\Omega\setminus \Gamma_{\epsilon/4}})$. We conclude that $v \in C^{0,1}(\overline{\Omega\setminus \Gamma_{\epsilon/4}})$. Together with the fact that $v \in C^{0,1}(\overline{\Gamma_{\epsilon/2}})$ we obtain that $v \in C^{0,1}(\overline{\Omega})$. 
\end{proof}

With our additional smoothness requirements  we can also achieve $BV$-regularity. 

\begin{cor}[$BV(\Omega)$-regularity] \label{cor:BVreg}
Suppose that $\Gamma = \partial \Omega' \in C^2$ and $Q \in W^{2,p}(\Omega)$, $p > n$. Then the solution $v$ of \eqref{eq:measdir} satisfies $\nabla v \in BV( \Omega; \mathbb{R}^n)$. 
\end{cor}
\begin{proof} Let $\epsilon > 0$ be as in the previous corollary.
We have seen in the proof of the previous corollary that $v + \frac{Q}{2} |d|$ lies in $W^{2,p}_{loc}(\Gamma_\epsilon)$ and hence $ \nabla \left( v + \frac{Q}{2} |d| \right) \in W^{1,p}_{loc} ( \Gamma_\epsilon) \subset  BV(\Gamma_{\epsilon/2})$. To show that $\nabla v \in BV(\Gamma_{\epsilon/2})$ it is hence sufficient to prove that $\nabla (\frac{Q}{2}|d|) \in BV(\Gamma_{\epsilon/2})$. As in Lemma \ref{lem:regudist} we decompose $\Gamma_\epsilon = \Gamma_\epsilon^+ \cup  \Gamma  \cup \Gamma_\epsilon^-$ and compute for $i \in \{ 1,...,n\}$ and arbitrary $\phi \in C^1_0(\Gamma_\epsilon; \mathbb{R}^n)$ such that $||\phi||_\infty \leq 1$ 
\begin{align}
\int_{\Gamma_\epsilon} \partial_i & \left(  \frac{Q}{2} |d| \right) \mathrm{div}( \phi) \dx  = \int_{\Gamma_{\epsilon}^+} \partial_i \left( \frac{Q}{2} d \right) \mathrm{div} ( \phi)  \dx - \int_{\Gamma_{\epsilon}^-} \partial_i \left( \frac{Q}{2} d \right) \mathrm{div} ( \phi)  \dx 
\\ & = 2 \int_{\Gamma}  \partial_i \left( \frac{Q}{2} d \right) \phi \mathrm{d}\mathcal{H}^{n-1} - \int_{\Gamma_\epsilon^+} \nabla \partial_i \left( \frac{Q}{2} d \right) \phi \dx  + \int_{\Gamma_\epsilon^-} \nabla  \partial_i \left( \frac{Q}{2} d \right) \phi \dx.
\end{align}
By the product rule one has $\frac{Q}{2} d \in W^{2,p}(\Gamma_\epsilon)$ and hence Sobolev embedding yields that $\left\Vert \partial_i (\frac{Q}{2}d)  \right\Vert_{L^\infty(\Gamma)} \leq C$ for some $C > 0$. Moreover
$\left\Vert\nabla \partial_i (\frac{Q}{2} d )   \right\Vert_{L^p(\Gamma_\epsilon)} \leq D$ for some $D > 0$. We infer that 
\begin{equation}
\int_{\Gamma_\epsilon} \partial_i \left( \frac{Q}{2} |d| \right) \mathrm{div}( \phi) \dx \leq 2 C \mathcal{H}^{n-1}( \Gamma)  +  D (|\Gamma_\epsilon^+|^\frac{1}{q} + |\Gamma_{\epsilon}^-|^\frac{1}{q}),
\end{equation}
where $q \in ( 1, \frac{n}{n-1} )$ is chosen such that $\frac{1}{p} + \frac{1}{q} =1$.
We have shown that $\nabla (\frac{Q}{2} |d|) \in BV( \Gamma_\epsilon)$ and thus also $\nabla v \in BV(\Gamma_{\epsilon/2})$. Since $v$ is harmonic on $\Omega \setminus \Gamma_{\epsilon/4}$ and $v\vert_{\partial ( \Omega \setminus \Gamma_{\frac{\epsilon}{4}} ) }$ is smooth we have that $v \in W^{2,2}(\Omega \setminus \Gamma_{\epsilon/4})$ and hence $\nabla v \in BV( \Omega \setminus \Gamma_{\epsilon/4})$. We infer that $\nabla v \in BV(\Omega)$ by the gluing property, cf. \cite[Remark 2.14]{Giusti}.  
\end{proof}

\subsection{Blow-up arguments}

We have now obtained Lipschitz-regularity for smooth domains and sufficiently smooth data $Q$. To pass to less regular settings we argue by approximation. To this end we discuss a-priori estimates that can be obtained by looking at the precise behavior of Lipschitz solutions on $\Gamma$. 

\begin{lemma}[A Taylor expansion]
Suppose that $\Gamma = \partial \Omega' \in C^1$, $x_0  \in \Gamma$, $Q \in C^0(\Gamma)$ and $u \in C^{0,1}(\overline{\Omega})$ is a solution of \eqref{eq:measdir}. Then there exists a vector $\theta(x_0) \in \mathbb{R}^n$ such that in a neighborhood of $x_0$ one has
\begin{equation}\label{eq:Taylor}
u(x) = u(x_0) + \theta(x_0) \cdot (x-x_0) - \frac{1}{2} Q(x_0) |(x- x_0, \nu(x_0))| +o (|x-x_0|). 
\end{equation} 
\end{lemma}
\begin{proof}
 Define $u_r : \mathbb{R}^n \rightarrow \mathbb{R}$ to be 
\begin{equation}
u_r(x) := \begin{cases} 
\frac{u(x_0+rx) - u(x_0)}{r}  & x_0 + r x \in \Omega, \\ \frac{-u(x_0)}{r}  & \mathrm{otherwise.}
\end{cases} 
\end{equation}
Note that $u_r \in C^{0,1}(\mathbb{R}^n)$ and $||\nabla u_r ||_{L^\infty(\mathbb{R}^n)} \leq ||\nabla u ||_{L^\infty(\Omega)}$. Moreover $u_r(0)= 0$, so $|u_r(z)|\leq ||\nabla u||_{L^\infty} |z|$ for all $z \in \mathbb{R}^n$ and $r > 0$. We conlude by Proposition \ref{prop:C01approx} that there exists a subsequence $r_j \rightarrow 0 $ and some $\overline{u} \in C^{0,1}_{loc}( \mathbb{R}^n)$ with $\nabla \overline{u} \in L^\infty(\mathbb{R}^n)$ such that 
$
u_{r_j} \rightarrow  \overline{u} 
$
weakly in $W^{1,2}(B_R(0))$ and uniformly on $B_R(0)$ for all $R > 0$. Next let $\phi \in C_0^\infty(\mathbb{R}^n)$ be fixed. Let $R_0 > 0$ be such that $\mathrm{supp}(\phi) \in B_{R_0}(0)$ and $j_0 \in  \mathbb{N}$ be such that $\frac{1}{r_j} (\Omega - x_0) \supset B_{R_0}(0)$ for all $j \geq j_0$. Next we derive a PDE for $\overline{u}$. 
We compute
\begin{align}
&\int_{\mathbb{R}^n} \nabla \overline{u} \nabla \phi \dx  = \int_{B_{R_0}(0)} \nabla \overline{u} \nabla \phi\;  \mathrm{d}x = \lim_{j \rightarrow \infty, j \geq j_0} \int_{B_{R_0}(0)} \nabla u_{r_j} \nabla \phi \dx 
\\ &= \lim_{j \rightarrow \infty, j \geq j_0}\int_{B_{R_0}(0)} \nabla u (x_0 + r_j x)  \nabla \phi(x)  \dx  = \lim_{j \rightarrow \infty, j \geq j_0}\int_{\frac{1}{r_j}(\Omega - x_0)} \nabla u (x_0 + r_jx)  \nabla \phi(x)  \dx  
\\ &= \lim_{j \rightarrow \infty}  \frac{1}{r_j^n}  \int_\Omega  \nabla u (y) \nabla \phi \left( \frac{y-x_0}{r_j} \right) \dy
 = \lim_{j \rightarrow \infty} \frac{1}{r_j^{n-1}} \int_\Omega \nabla u(y) \nabla \left( \phi \left( \frac{(\cdot) - x_0}{r_j} \right) \right)  (y) \dy 
\\ & = \lim_{j \rightarrow \infty} \frac{1}{r_j^{n-1}} \int_\Gamma Q(z) \phi \left( \frac{z- x_0}{r_j} \right) d\mathcal{H}^{n-1} (z) 
= \lim_{j \rightarrow \infty} \frac{1}{r_j^{n-1}} \int_{\Gamma \cap B_{r_j R_0} (x_0)}  Q(z) \phi \left( \frac{z- x_0}{r_j} \right) d\mathcal{H}^{n-1} (z) .
\end{align}
Now we can find an orthogonal matrix $R\in O(n)$ and $t > 0$  such that  $\Gamma \cap B_{r_j R_0} (x_0)  \subset  x_0 +  R \{ (x',f(x'))^T  : x' \in B_t(0) \}$ for sufficiently large $j$. Here $f \in C^1(\overline{B_{t_0}(0)})$ satisfies $f(0) = 0 $ and $\nabla f(0) = 0$. We can also achieve that $R e_n = \nu_{\Omega'} (x_0)$. Hence 
\begin{align}
\int_{\mathbb{R}^n} \nabla \overline{u} \nabla \phi \dx & = \lim_{j \rightarrow \infty}  \frac{1}{r_j^{n-1}}  \int_{B_t(0)} Q(x_0 + R(z', f(z'))^T ) (\phi \circ R) \left(  \frac{z'}{r_j}, \frac{f(z')}{r_j}  \right) \sqrt{1+ |\nabla f(z')|^2} dz'
\\ & = \lim_{j \rightarrow \infty}  \int_{ B_{\frac{t}{r_j}}(0) } Q(x_0 + R (r_j s, f(r_j s))^T) (\phi \circ R) ( s, \frac{1}{r_j} f(r_j s))  \sqrt{1 + |\nabla f(r_j s)|^2}  \ds.
\end{align}
Since $ \left\vert R (s, \frac{1}{r_j} f(r_j s))^T \right\vert \geq |s|$ we find that $|s| \geq  R_0$ implies  $ R (s, \frac{1}{r_j} f(r_j s))^T  \not \in \mathrm{supp}(\phi)$. Using the dominated convergence theorem, $\nabla f(0) = 0$, and  $\mathrm{supp}(\phi) \subset B_{R_0} (0)$ we find
\begin{align}
\int_{\mathbb{R}^n} \nabla \overline{u} \nabla \phi \dx&  = \lim_{j \rightarrow \infty}  \int_{ B_{R_0}(0) } Q(x_0 + R (r_j s, f(r_j s))^T) (\phi \circ R) ( s, \frac{1}{r_j} f(r_j s))  \sqrt{1 + |\nabla f(r_j s)|^2}  \ds
\\ & = Q(x_0) \int_{B_{R_0}(0)} (\phi\circ R )(s, \nabla f(0) \cdot s ) \; \ds = Q(x_0) \int_{\mathbb{R}^{n-1}} (\phi \circ R)(s,0) \ds
\end{align}
Summarizing our findings we have that $\overline{u}$ satisfies $\overline{u} \in C^{0,1}$, $\nabla \overline{u} \in L^\infty(\mathbb{R}^n)$, $\overline{u}(0)= 0$ and for all $\phi \in C_0^\infty(\mathbb{R}^n)$ one has 
\begin{equation}\label{eq:ubareq}
\int_{\mathbb{R}^n} \nabla \overline{u} \nabla \phi \dx  = Q(x_0) \int_{\mathbb{R}^{n-1}} (\phi\circ R)(s,0) \ds .
\end{equation}
Next we define $\widetilde{u} : \mathbb{R}^n \rightarrow \mathbb{R}$ by $\widetilde{u} (x) := - \frac{1}{2}|e_n^T R^T x| = - \frac{1}{2}| (x , \nu_{\Omega'}(x_0))|$. We will next understand the relation between $\overline{u}$ and $\widetilde{u}$. Set $H^{+} := \{ z \in \mathbb{R}^n : e_n^T z > 0 \}$, $H^{-} := \{ z \in \mathbb{R}^n : e_n^T z < 0 \}$. One readily checks that $\widetilde{u} \in C^{0,1}_{loc}, \nabla \widetilde{u} \in L^\infty(\mathbb{R}^n)$, $\widetilde{u}(0) = 0$ and for $\phi \in C_0^\infty( \mathbb{R}^n)$ one computes 
\begin{align}
-2\int_{\mathbb{R}^n} \nabla \widetilde{u} \nabla \phi \dx & = \int_{RH^+} (Re_n)^T \nabla \phi  \dx - \int_{R H^{-}} (Re_n)^T \nabla \phi  \dx  \\
& = \int_{H^+} (Re_n)^T (\nabla \phi) (Rx) \dx  - \int_{H^-} (Re_n)^T (\nabla \phi)(Rx) \dx 
\\ & = \int_{H^+} (Re_n)^T R \nabla(\phi \circ R)  \dx  - \int_{H^-} (Re_n)^T R \nabla (\phi \circ R) \dx
\\ & = \int_{H^+} \partial_n (\phi \circ R) \dx - \int_{H^-}  \partial_n ( \phi \circ R ) \dx = -2\int_{\mathbb{R}^{n-1}} ( \phi \circ R) (s,0) \ds, 
\end{align}
where we used Fubini's theorem in the last step and performed the $x_n$-integration first.
We infer from this and \eqref{eq:ubareq}
that $\overline{u} -Q(x_0) \widetilde{u} $ lies in $C^{0,1}_{loc}(\mathbb{R}^n), \nabla (\overline{u}- Q(x_0)  \widetilde{u}) \in L^\infty( \mathbb{R}^n)$, $ (\overline{u}- Q(x_0) \widetilde{u}) (0 ) = 0$ and for all $\phi \in C_0^\infty(\mathbb{R}^n)$ one has 
\begin{equation}
\int_{\mathbb{R}^n} \nabla (\overline{u}-Q(x_0) \widetilde{u}) \nabla \phi = 0 .
\end{equation}
Hence $\overline{u}-Q(x_0) \widetilde{u}$ is a harmonic function on $\mathbb{R}^n$. This implies that also $\nabla ( \overline{u}-Q(x_0) \widetilde{u})$ is harmonic on $\mathbb{R}^n$. Since also $\nabla ( \overline{u} - Q(x_0) \widetilde{u} ) \in L^\infty(\mathbb{R}^n) $ we infer by Liouville's Theorem that $\nabla ( \overline{u} - Q(x_0) \widetilde{u} ) = \mathrm{const} =: \theta (x_0) \in \mathbb{R}^n$. We infer that 
$\overline{u} (x) = Q(x_0) \widetilde{u}(x) + \theta(x_0) x + b$ for some $b \in \mathbb{R}$, but since $\overline{u}(0) =  \widetilde{u}(0) =  0$ we find $b = 0$. Hence $\overline{u}(x) =- \frac{1}{2} Q(x_0) |(x, \nu_{\Omega'} (x_0) )| + \theta (x_0) x.$ Recalling the definition of $\overline{u}$ we obtain for all $x \in B_1(0)$
\begin{equation}
\lim_{r \rightarrow 0} \frac{u(x_0 + rx) - u(x_0)}{r} = \overline{u}(x) = \theta(x_0) x - \frac{1}{2} Q(x_0) |(x, \nu_{\Omega'} (x_0) )|.   
\end{equation}
We could as well write 
\begin{equation}
u(x_0 + r x ) = u(x_0) + \theta(x_0) (rx) - \frac{1}{2} Q(x_0) |(rx , \nu_{\Omega'} (x_0) ) | + o(r),
\end{equation}
which implies \eqref{eq:Taylor}.

\end{proof}

\begin{lemma}\label{lem:alphalim}
Let $\Gamma = \partial \Omega' \in C^1$ and $u \in C^{0,1}(\overline{\Omega})$ be a solution of \eqref{eq:measdir}. Further let $\theta: \Gamma \rightarrow \mathbb{R}^n, x_0 \mapsto \theta(x_0)$, where $\theta(x_0)$ is as in the previous Proposition. Then one has for all $z \in \Gamma$
\begin{equation}\label{eq:alphalim}
\theta(z) = \lim_{ r \rightarrow 0 } \fint_{B_r(z)}  \nabla u \; \mathrm{d}x.
\end{equation}
In particular, $\theta$ is measurable and lies in $L^\infty(\Gamma)$, satisfying $||\theta||_{L^\infty(\Gamma)} \leq ||\nabla u||_{L^\infty(\Omega)}$. 

\end{lemma}
\begin{proof}
 To show the measurabilty of $\theta$ and the $L^\infty$-estimate it suffices to prove \eqref{eq:alphalim}.  To this end we compute for all $z \in \Gamma$ and $r > 0$ small enough
\begin{align}
 \fint_{B_r(z)} \partial_i u \dx &   = \frac{1}{\alpha_n r^{n}} \int_{\partial B_r(z)}  u(x)  \frac{x_i-z_i}{r} \; \mathrm{d}\mathcal{H}^{n-1}(x) 
\\ &  = \frac{1}{\alpha_n r^{n}} \int_{\partial B_r(z)} ( u(x) - u(z))   \frac{x-z}{r} \; \mathrm{d}\mathcal{H}^{n-1}(x) 
\\ & =  \frac{1}{\alpha_n r^{n}} \int_{\partial B_r(z)} ( \theta(z)\cdot  (x-z)  -\frac{1}{2} Q(x_0) |(x-z, \nu_{\Omega'}(x_0))| + o(|x-z|) )  \frac{x-z}{r} \; \mathrm{d}\mathcal{H}^{n-1}(x).\label{eq:sonne} 
\end{align}
Now note that 
\begin{equation}
 \left\vert \frac{1}{\alpha_n r^n} \int_{\partial B_r(z) } o(|x-z|) \frac{x-z}{r} \; \mathrm{d} \mathcal{H}^{n-1}(x) \right\vert \leq \frac{ o(r) \omega_n r^{n-1} }{\alpha_n r^n }  \rightarrow 0,  \quad ( r \rightarrow 0 )  ,
\end{equation}
and 
\begin{equation}
\int_{\partial B_{r}(z) } \frac{1}{2} Q(x_0) |( x- z, \nu_{\Omega'} (x_0) ) | \frac{x-z}{r} \; \mathrm{d}\mathcal{H}^{n-1}(x) = 0 , 
\end{equation}
since the integrand is antisymmetric and $\partial B_r(z)$ is a symmetric set. Again using such reflection and symmetry arguments 
we obtain for $i \in \{ 1, ..., n \}$
\begin{align}
  \frac{1}{\alpha_n r^n }&  \int_{\partial B_r(z)} \theta(z) \cdot (x-z)  \frac{x_i-z_i}{r} \; \mathrm{d}\mathcal{H}^{n-1}(x)  \\ & =  \frac{1}{\alpha_n r^n }  \sum_{ j = 1}^n \int_{\partial B_r(z)} \theta_j(z)  (x_j-z_j)  \frac{x_i-z_i}{r} \; \mathrm{d} \mathcal{H}^{n-1}(x)  \\ & =  \frac{1}{\alpha_n r^n } \int_{\partial B_r(z) } \theta_i(z) \frac{(x_i - z_i)^2 }{r} \; \mathrm{d} \mathcal{H}^{n-1}(x) \\ & =  \frac{1}{n \alpha_n r^n } \int_{\partial B_r(z)} \theta_i(z)  \frac{|x-z|^2}{r} \; \mathrm{d}\mathcal{H}^{n-1}(x) =  \frac{1}{n \alpha_n r^n } \frac{r^2}{r} \theta_i(z) \omega_n r^{n-1} = \theta_i (z) ,
\end{align}
since $\omega_n = n \alpha_n$. The claim follows passing to the limit in \eqref{eq:sonne}. 
\end{proof}

\subsection{An a priori estimate}

Next we obtain an a priori estimate for $||\nabla u||_{L^\infty(\Omega)}$ via the maximum principle on $\Omega'$ and $ \Omega'' := \Omega \setminus \overline{\Omega'}$. We shall use the notation for $\Omega''$ in the entire rest of this section.

The regularity that we have proved so far is however not  sufficient to apply the maximum principle to $\nabla u$ -- neither for the classical nor for the weak (Sobolev) maximum principle. Hence we need the following generalization. 


\begin{lemma}[A maximum principle for $BV$-solutions, Proof in Appendix \ref{app:maxpr}] \label{lem:maxpr}
Let $U\subset \mathbb{R}^n$ be open and bounded with $C^\infty$-smooth boundary and $w \in BV(U)$ be such that 
\begin{equation}\label{eq:weakharm}
\int_U w \Delta \phi = 0 \quad \forall \phi \in C_0^\infty( U).  
\end{equation}
 Then $||w||_{L^\infty(U)} \leq ||\mathrm{tr}_U(w)||_{L^\infty(\partial U)}$, where  $\mathrm{tr}_{U}(w)$ denotes the $BV(U)$-trace of $w$. 
\end{lemma}

To the best of our knowledge this result is not known for $BV$-solutions but only for $W^{1,1}$-solutions. This is why we give a proof in the appendix.  
Next we compute the $BV(\Omega')$-trace of $\nabla u$.

\begin{lemma}\label{lem:tronGamma}
Let $\Gamma = \partial \Omega'    \in C^{1}$, $\Omega'' := \Omega \setminus \overline{\Omega'}$, and $u \in C^{0,1}(\overline{\Omega})$ be a solution of \eqref{eq:measdir} with $Q \in C^0(\Gamma)$ such that $\nabla u \in BV(\Omega)$. Then one has $\mathcal{H}^{n-1}$ a.e. on $\Gamma$ 
\begin{equation}
\mathrm{tr}_{\Omega'} (\nabla u) = \theta + \frac{Q}{2} \nu_{\Omega'}
\end{equation}
\begin{equation}
\mathrm{tr}_{\Omega''} ( \nabla u ) = \theta -\frac{Q}{2} \nu_{\Omega'},
\end{equation}
where $\theta$ is as in \eqref{eq:Taylor} and \eqref{eq:alphalim}. 
\end{lemma} 
\begin{proof}
We only show the first equality, the other one is completely analogous. We know by Lemma \ref{lem:BVtr} that $\mathcal{H}^{n-1}$ a.e. one has 
\begin{align}
\mathrm{tr}_{\Omega'} (\partial_i u)(z) = \lim_{r \rightarrow 0} \fint_{B_r(z) \cap \Omega'} \partial_i  u (x)  \dx.
\end{align}
We compute using \eqref{eq:Taylor} and the fact that by Proposition \ref{prop:transverselip} $B_r(z) \cap \Omega'$ is a Lipschitz domain for $r$ small enough.
\begin{align}
& \mathrm{tr}_{\Omega'} (\partial_i u)(z)   = \lim_{r \rightarrow 0 } \frac{1}{|B_r(z) \cap \Omega'|}\int_{\partial (B_r(z) \cap \Omega')}   u(y) \nu^i(y)  \; \mathrm{d} \mathcal{H}^{n-1}(y) 
\\ & =  \lim_{r \rightarrow 0 } \frac{1}{|B_r(z) \cap \Omega'|}\left( \int_{\partial (B_r(z) \cap \Omega')}   (u(y)- u(z)) \nu^i(y)  \; \mathrm{d} \mathcal{H}^{n-1}(y) + u(z) \int_{\partial (B_r(z) \cap \Omega')} \nu^i(y)  \; \mathrm{d} \mathcal{H}^{n-1}(y) \right) 
\\ & = \lim_{r \rightarrow 0 } \frac{1}{|B_r(z) \cap \Omega'|} \left( \int_{\partial (B_r(z) \cap \Omega')}   (u(y)- u(z)) \nu^i(y)  \; \mathrm{d} \mathcal{H}^{n-1}(y) + u(z) \int_{B_r(z) \cap \Omega'} \partial_i (1) \; \mathrm{d}y \right) 
\\ & = \lim_{r \rightarrow 0 } \frac{1}{|B_r(z) \cap \Omega'|}  \int_{\partial (B_r(z) \cap \Omega')}  ( \theta(z)\cdot(y-z) - \frac{1}{2} Q(z) |(y-z, \nu(z))|  + o(|y-z|)) \nu^i(y) ) \; \mathrm{d} \mathcal{H}^{n-1}(y).
\end{align}
Note that
\begin{equation}
\lim_{r \rightarrow 0 } \frac{1}{|B_r(z) \cap \Omega'|}  \left\vert \int_{\partial (B_r(z) \cap \Omega')} o(|y-z|) \; \mathrm{d}\mathcal{H}^{n-1}(y)  \right\vert  \leq \limsup_{r \rightarrow 0 } o(1) \frac{r \mathcal{H}^{n-1}(\partial (B_r(z) \cap \Omega')) }{| B_r(z) \cap \Omega'|}.
\end{equation}
Now we obtain by \cite[Theorem 5.14(iii) and Theorem 5.15]{EvGar} that $\mathcal{H}^{n-1}(\partial \Omega' \cap \overline{B_r(z)} ) \leq (2^{n-1} + o(1)) \alpha_{n-1} r^{n-1}$. Moreover $\mathcal{H}^{n-1}(\overline{\Omega'} \cap \partial B_r(z)) \leq  \omega_n r^{n-1} = n \alpha_n r^{n-1}$. Therefore 
\begin{align}\label{eq:o(1)sclaing}
& \limsup_{r \rightarrow 0 } o(1) \frac{r \mathcal{H}^{n-1}(\partial (B_r(z) \cap \Omega') }{| B_r(z) \cap \Omega'|}   \leq \limsup_{r \rightarrow 0 } o(1) \frac{r (\mathcal{H}^{n-1}(\overline{B_r(z)} \cap \partial\Omega') + \mathcal{H}^{n-1} ( \overline{\Omega'} \cap \partial B_r(z)) } {|B_r(z) \cap \Omega'| }
\\ & = \limsup_{r \rightarrow 0 } o(1) \frac{r^n   (n + 2^{n-1} + o(1) ) \alpha_{n-1}}{|B_r(z) \cap \Omega' | }  = \limsup_{r \rightarrow 0 } o(1)  \frac{(n + 2^{n-1} + o(1) ) \alpha_{n-1} }{|B_1(0) \cap  ( \frac{\Omega' - z}{r} ) |} = 0 ,
\end{align} 
where we used that by \cite[Theorem 5.13]{EvGar} 
\begin{equation}
\lim_{r \rightarrow 0 } \frac{1}{|B_1(0) \cap  ( \frac{\Omega' - z}{r} ) |} = \frac{1}{|B_1(0) \cap H^-| }  \in (0, \infty) ,
\end{equation}
where $H^- = \{ y \in \mathbb{R}^n : \nu(z) \cdot y < 0 \} $. For later use we also define $H^{+} :=  \{ y \in \mathbb{R}^n : \nu(z) \cdot y > 0 \}$. 
Now using once more the divergence theorem we obtain
\begin{equation}
\frac{1}{|B_r(z) \cap \Omega'|}  \int_{\partial (B_r(z) \cap \Omega')}   \theta(z)\cdot(y-z) \nu^i(y) \mathrm{d} \mathcal{H}^{n-1}(y) = \frac{1}{|B_r(z) \cap \Omega'| } \int_{B_r(z) \cap \Omega'} \theta_i(z)  = \theta_i(z),
\end{equation}
for all $r > 0$. Hence we obtain  that $\mathcal{H}^{n-1}$-a.e. one has
\begin{align}
\mathrm{tr}_{\Omega'} (\partial_i u)(z) &  =   \theta_i(z) - \lim_{r \rightarrow 0 }\frac{1}{|B_r(z) \cap \Omega'|}  \int_{\partial( B_r(z) \cap \Omega')} \frac{Q(z)}{2}  |(y-z, \nu(z))| \nu^i(y) \; \mathrm{d}\mathcal{H}^{n-1}(y). 
\\ & =  \theta_i(z) -  \lim_{r \rightarrow 0 }    \frac{1}{|B_r(z) \cap \Omega'|}   \frac{Q(z)}{2} \int_{\partial( B_r(z) \cap \Omega')}  |(y-z, \nu(z))| \nu^i(y) \; \mathrm{d}\mathcal{H}^{n-1}(y). 
\end{align}
Therefore
\begin{align}
& \mathrm{tr}_{\Omega'} ( \partial_i u) (z)   = \theta_i(z) -  \lim_{r \rightarrow 0 }  \frac{1}{|B_r(z) \cap \Omega'|} \frac{Q(z)}{2}  \int_{ B_r(z) \cap \Omega'}   \partial_{y_i} |(y-z, \nu(z))| \dy 
\\ & = \theta_i(z) - \frac{Q(z)}{2} \lim_{r \rightarrow 0} \frac{|B_r(z) \cap \Omega' \cap H^+|- |B_r(z) \cap \Omega' \cap H^-| }{|B_r(z) \cap \Omega'| }
 \\ & = \theta_i(z) - \frac{Q(z)\nu^i(z)}{2} \lim_{r \rightarrow 0} \frac{|B_1(0) \cap \frac{\Omega'-z}{r} \cap H^+|- |B_1(0) \cap \frac{\Omega'-z}{r} \cap H^-| }{|B_1(0) \cap \frac{\Omega'-z}{r} | } \\& = \theta_i(z) - \frac{Q(z)\nu^i(z)}{2} \frac{|B_1(0) \cap H^- \cap H^+| - |B_1(0) \cap H^{-}|}{|B_1(0) \cap H^-|} = \theta_i(z) + \frac{Q(z) \nu^i(z)}{2},
\end{align}
where we used \cite[Theorem 5.13]{EvGar} in the last step. 
\end{proof}

In a nonstandard sense we can hence look at $\nabla u \in BV(\Omega')$ as a solution of the following Dirichlet problem
\begin{equation}
\begin{cases} 
\Delta ( \nabla u) = 0 & \mathrm{in} \; \Omega' \\
\nabla u = \theta +  \frac{1}{2} Q \nu   & \mathrm{on} \; \partial \Omega' , \end{cases} \quad \quad \quad 
\end{equation} 
where the last line holds in the sense of $BV$-traces. In the same manner $\nabla u \in BV(\Omega'')$ solves
\begin{equation}
\begin{cases} 
\Delta ( \nabla u) = 0 & \mathrm{in} \; \Omega'' \\
\nabla u = \theta -  \frac{1}{2} Q \nu   & \mathrm{on} \; \partial \Omega'' \cap \partial \Omega' . \end{cases}
\end{equation}
Notice that in the sense of $BV$-traces $\nabla u$ makes a \emph{normal jump} on $\Gamma = \partial \Omega'= \partial \Omega'' \cap \partial \Omega'$. This can be seen as a new version of the potential theoretic statement \eqref{eq:normaljump}, which characterizes the normal jump of the single-layer potential.

Having this characterization of $\nabla u$ at hand we can apply the ($BV$-)maximum principle to estimate $||\nabla u||_{L^\infty(\Omega)}$.

\begin{cor}
Let $\Gamma = \partial \Omega'  \in C^{\infty}$ and $u \in C^{0,1}(\overline{\Omega})$ be a solution of \eqref{eq:measdir} with $Q \in C^0(\Gamma)$ such that $\nabla u \in BV(\Omega)$. Then one has
\begin{equation}\label{eq:cor48}
||\theta||_{L^\infty(\Gamma)} \leq ||\nabla u ||_{L^\infty(\Omega)} \leq ||\theta ||_{L^\infty(\Gamma)} + \frac{1}{2} ||Q||_{L^\infty(\Gamma)}  + C(\Omega, \mathrm{dist}(\Gamma, \partial \Omega)) ||u||_{L^{\infty}}.
\end{equation}
The constant $C(\Omega,D)$ can always be chosen to be increasing in $D+ \frac{1}{D}$.
\end{cor}
\begin{proof}
Let $u$ be as in the statement. 
The first inequality has already been shown in Lemma \ref{lem:alphalim}. For the second inequality observe that $\nabla u$ is harmonic in $\Omega'$ and in $\Omega''$ and hence by Lemma \ref{lem:maxpr}
\begin{align}
||\nabla u||_{L^\infty(\Omega)} & \leq \max \{  || \nabla u ||_{L^\infty(\Omega')} , || \nabla u ||_{L^\infty(\Omega'')} \} \leq \max \{ ||\mathrm{tr}_{\Omega'} ( \nabla u ) ||_{L^\infty(\partial \Omega')}, ||\mathrm{tr}_{\Omega''} ( \nabla u ) ||_{L^\infty(\partial \Omega'')} \}
\\ & = \max\{ ||\mathrm{tr}_{\Omega'} ( \nabla u ) ||_{L^\infty(\Gamma)}, ||\mathrm{tr}_{\Omega''} ( \nabla u ) ||_{L^\infty(\Gamma \cup \partial \Omega)} \}. \label{eq:trace2} 
\end{align}
On $\Gamma$ one has by Lemma \ref{lem:tronGamma} that 
\begin{equation}\label{eq:trace1} 
||\mathrm{tr}_{\Omega'} ( \nabla u ) ||_{L^\infty(\Gamma)}, ||\mathrm{tr}_{\Omega''} ( \nabla u ) ||_{L^\infty(\Gamma)} \leq ||\theta||_{L^\infty(\Gamma)} + \frac{1}{2} || Q ||_{L^\infty(\Gamma)}. 
\end{equation}
It only remains to estimate $||\mathrm{tr}_{\Omega''}(\nabla u) ||_{L^\infty(\partial \Omega)}$. To this end let $\delta = \min\{ \delta_0 ,  \frac{1}{2}\mathrm{dist}(\Gamma, \partial\Omega) \}$, where $\delta_0 > 0$ is such that for all $\epsilon< 2\delta_0$ one has that $\Omega^\epsilon := \{ x \in \Omega : \mathrm{dist}(x, \partial\Omega) < \epsilon \}$ is a $C^{2,\gamma}$- domain. Since $u$ is harmonic on $\Omega^\delta$ and takes smooth values on $\partial \Omega^\delta$ we infer by Schauder theory (cf. \cite[Theorem 2.19]{Sweers}) that $u \in C^{2,\gamma}(\overline{\Omega}^\delta)$ and by \cite[Theorem 8.33]{GilTru} we have that 
\begin{equation}
||u||_{C^{1,\gamma}(\overline{\Omega^\delta})} \leq C_0({\Omega^{\delta}}) ||u||_{L^\infty}, 
\end{equation} 
where the constant $C_0({\Omega^\delta})$ depends only on the measure of $\Omega^\delta$ (or better on $|\Omega^\delta| + \frac{1}{|\Omega^\delta|}$) and the boundary parametrizations of $\Omega^\delta$. One should notice now that this constant does not depend on $\Gamma$ but only on $\Omega$ and $\delta$. Using the choice of $\delta$ we find 
\begin{equation}
||\nabla u||_{L^\infty(\partial \Omega) } \leq C( \Omega, \mathrm{dist}( \Gamma, \partial \Omega)) ||u||_{L^\infty},
\end{equation}
where $C(\Omega, D)$ depends increasingly on $D + \frac{1}{D}$.
This, \eqref{eq:trace1} and \eqref{eq:trace2} imply the claim. \end{proof}

\begin{lemma} [An a priori estimate for $||\theta||_{L^\infty(\Gamma)}$]\label{lem:49}
Let $\Gamma = \partial \Omega' \in C^\infty$, $u \in C^{0,1}(\overline{\Omega})$ be as in the previous corollary. Then there exists $C= C(n) > 0$ such that 
\begin{equation}
||\theta||_{L^\infty(\Gamma)} \leq C(n) ||Q||_{L^\infty(\Gamma)}.
\end{equation} 
\end{lemma} 
\begin{proof}
Let $x_0 \in \Gamma$. 
Let $\psi \in C^\infty_0 ([0,1))$ be such that $\psi \vert_{[0, \frac{1}{4}]} = 1$, $\psi \vert_{[\frac{3}{4}, 1] } = 0$, $0 \leq \psi \leq 1$ and $\psi' \leq 0$. We test \eqref{eq:measdir} with $\phi_r(x) := \theta(x_0) \cdot (x- x_0) \psi( \frac{|x-x_0|}{r})$  for $r \in  (0, \frac{1}{2} \mathrm{dist}(\Gamma, \partial \Omega)) $.
We have to perform some integrations by parts, even more than one would guess, since the Taylor expansion \eqref{eq:Taylor} is only of first order. We conclude
\begin{align}\label{eq:maincompu}
& \int_\Gamma \phi_r Q \; \mathrm{d}\mathcal{H}^{n-1}   = \int_{B_r(x_0)} \nabla u(x) \cdot \left( \theta(x_0) \psi\left(\frac{|x-x_0|}{r} \right)  + \theta(x_0) \frac{|x-x_0|}{r} \psi'\left(\frac{|x-x_0|}{r}\right) \right) \dx
\\ & =  \int_{\partial B_r(x_0) } u(x)  \left( \theta(x_0) \psi\left(\frac{|x-x_0|}{r} \right)  + \theta(x_0) \frac{|x-x_0|}{r} \psi'\left(\frac{|x-x_0|}{r}\right) \right) \cdot \nu_{B_r(x_0)} \; \mathrm{d}\mathcal{H}^{n-1}(x) \\ & - \int_{ B_r(x_0) } u (x)  \;  \mathrm{div}\left( \theta(x_0) \psi\left(\frac{|x-x_0|}{r} \right)  + \theta(x_0) \frac{|x-x_0|}{r} \psi'\left(\frac{|x-x_0|}{r}\right) \right) \dx 
\\  & =  - \int_{ B_r(x_0) }  u (x)  \;  \mathrm{div}\left( \theta(x_0) \psi\left(\frac{|x-x_0|}{r} \right)  + \theta(x_0) \frac{|x-x_0|}{r} \psi'\left(\frac{|x-x_0|}{r}\right) \right) \dx 
\\ & =  - \int_{ B_r(x_0) }  (u (x) - u(x_0))   \;  \mathrm{div}\left( \theta(x_0) \psi\left(\frac{|x-x_0|}{r} \right)  + \theta(x_0) \frac{|x-x_0|}{r} \psi'\left(\frac{|x-x_0|}{r}\right) \right) \dx
\\ & = - \int_{B_r(x_0)} \theta(x_0) \cdot (x-x_0)  \mathrm{div}\left( \theta(x_0) \psi\left(\frac{|x-x_0|}{r} \right)  + \theta(x_0) \frac{|x-x_0|}{r} \psi'\left(\frac{|x-x_0|}{r}\right) \right) \dx
\\ & + \int_{B_r(x_0)} \frac{1}{2} Q(x_0) |(x-x_0, \nu(x_0))|   \mathrm{div}\left( \theta(x_0) \psi\left(\frac{|x-x_0|}{r} \right)  + \theta(x_0) \frac{|x-x_0|}{r} \psi'\left(\frac{|x-x_0|}{r}\right) \right) \dx
\\ & - \int_{B_r(x_0)} o(|x-x_0|) \mathrm{div}\left( \theta(x_0) \psi\left(\frac{|x-x_0|}{r} \right)  + \theta(x_0) \frac{|x-x_0|}{r} \psi'\left(\frac{|x-x_0|}{r}\right) \right)  \dx .
\end{align}
Now note that 
\begin{align}
&\left\vert \mathrm{div}\left( \theta(x_0) \psi\left(\frac{|x-x_0|}{r} \right)  + \theta(x_0) \frac{|x-x_0|}{r} \psi'\left(\frac{|x-x_0|}{r}\right) \right) \right\vert  \\ & = \left\vert \frac{\theta(x_0)\cdot (x-x_0)}{r|x-x_0|} \left( 2\psi' \left( \frac{|x-x_0|}{r}  \right)  + \frac{|x-x_0|}{r} \psi''\left( \frac{|x-x_0|}{r} \right) \right) \right\vert
 \leq \frac{|\theta(x_0)| (2||\psi'||_{L^\infty} + || \psi''||_{L^\infty} )}{r}.
\end{align}
Hence 
\begin{align}
& \int_{B_r(x_0)} o(|x-x_0|) \mathrm{div}\left( \theta(x_0) \psi\left(\frac{|x-x_0|}{r} \right)  + \theta(x_0) \frac{|x-x_0|}{r} \psi'\left(\frac{|x-x_0|}{r}\right) \right)
\\ &  \quad = o(r) |B_r(x_0)| \frac{|\theta(x_0)| (2||\psi'||_{L^\infty} + || \psi''||_{L^\infty} )}{r} = o(r) \alpha_n r^{n-1} |\theta(x_0)| (2||\psi'||_{L^\infty} + || \psi''||_{L^\infty} )
\\ &= o(1) r^n |\theta(x_0)|   \label{eq:alphainfbere}.
\end{align}
Moreover integrating by parts once more (using that $\phi_r \in C_0^\infty(B_r(x_0)$ and $\nabla |(x,z)| = \mathrm{sgn}(x,z) z$ in the sense of weak derivatives) we find 
\begin{align}
&\int_{B_r(x_0)} \frac{1}{2} Q(x_0) |(x-x_0, \nu(x_0))|   \mathrm{div}\left( \theta(x_0) \psi\left(\frac{|x-x_0|}{r} \right)  + \theta(x_0) \frac{|x-x_0|}{r} \psi'\left(\frac{|x-x_0|}{r}\right) \right) \dx
\\ &  = - \int_{B_r(x_0)} \frac{1}{2} Q(x_0) \mathrm{sgn}((x-x_0, \nu(x_0)) \nu(x_0) \cdot \left( \theta(x_0) \psi\left(\frac{|x-x_0|}{r} \right)  + \theta(x_0) \frac{|x-x_0|}{r} \psi'\left(\frac{|x-x_0|}{r}\right) \right) \dx
\\ &=  -\int_{B_r(0)} \frac{1}{2} Q(x_0) \mathrm{sgn}((y, \nu(x_0)) \nu(x_0) \cdot \left( \theta(x_0) \psi\left(\frac{|y|}{r} \right)  + \theta(x_0) \frac{|y|}{r} \psi'\left(\frac{|y|}{r}\right) \right) \dx = 0 ,\label{eq:Qbere}
\end{align}
since the integrand is antisymmetric with respect to the transformation $y \mapsto - y$, though the set $B_r(0)$ is symmetric with respect to this transformation. 
Furthermore we infer
\begin{align}
& - \int_{B_r(x_0)} \Big[\theta(x_0) \cdot (x-x_0) \Big] \mathrm{div}\left( \theta(x_0) \psi\left(\frac{|x-x_0|}{r} \right)  + \theta(x_0) \frac{|x-x_0|}{r} \psi'\left(\frac{|x-x_0|}{r}\right) \right) \dx
\\ &
 = - \int_{B_r(x_0)} \Big[\theta(x_0)\cdot (x-x_0) \Big] \Big[ \theta(x_0) \cdot \left( 2 \psi' \left( \frac{|x-x_0|}{r} \right) \frac{x-x_0}{r|x-x_0|} + \psi'' \left( \frac{|x-x_0|}{r} \right) \frac{(x-x_0)}{r^2}  \right) \Big] \dx 
 \\ &
  = - \int_{B_r(x_0)} \Big[\theta(x_0)\cdot (x-x_0) \Big]^2 \left( 2 \psi' \left( \frac{|x-x_0|}{r} \right) \frac{1}{r|x-x_0|} + \psi'' \left( \frac{|x-x_0|}{r} \right) \frac{1}{r^2} \right) \dx
  \\ & 
  = - \int_{B_r(0)} \Big[\theta(x_0) \cdot z \Big]^2 \left( 2 \psi' \left( \frac{|z|}{r} \right) \frac{1}{r|z|} + \psi'' \left( \frac{|z|}{r} \right) \frac{1}{r^2} \right) \dz
  \\ & = - |\theta(x_0)|^2 \int_{B_r(0)} \Big[\frac{\theta(x_0)}{|\theta(x_0)|} \cdot z \Big]^2 \left( 2 \psi' \left( \frac{|z|}{r} \right) \frac{1}{r|z|} + \psi'' \left( \frac{|z|}{r} \right) \frac{1}{r^2} \right) \dz.
\end{align}
Applying an orthogonal tranformation that maps $\frac{\theta(x_0)}{|\theta(x_0)|}$ to $e_n$, the $n$-th unit vector, and using the symmetry of the expression with respect to the labeling of coordinates we find
\begin{align}
    & - \int_{B_r(x_0)} \theta(x_0) \cdot (x-x_0)  \mathrm{div}\left( \theta(x_0) \psi\left(\frac{|x-x_0|}{r} \right)  +\theta(x_0) \frac{|x-x_0|}{r} \psi'\left(\frac{|x-x_0|}{r}\right) \right) \dx
    \\ & 
    = - |\theta(x_0)|^2 \int_{B_r(0)} z_n^2 \left( 2 \psi' \left( \frac{|z|}{r} \right) \frac{1}{r|z|} + \psi'' \left( \frac{|z|}{r} \right) \frac{1}{r^2} \right) \dz.
    \\ & 
    = - \frac{|\theta(x_0)|^2}{n}  \int_{B_r(0)}(z_1^2 + ... + z_n^2) \left( 2 \psi' \left( \frac{|z|}{r} \right) \frac{1}{r|z|} + \psi'' \left( \frac{|z|}{r} \right) \frac{1}{r^2} \right) \dz.
    \\ &= 
    - \frac{|\theta(x_0)|^2}{n}  \int_{B_r(0)} \left( 2 \psi' \left( \frac{|z|}{r} \right) \frac{|z|}{r} + \psi'' \left( \frac{|z|}{r} \right) \frac{|z|^2}{r^2} \right) \dz
    \\ & = 
     - \frac{|\theta(x_0)|^2}{n} \int_0^r \omega_n s^{n-1}\left( 2 \psi' \left( \frac{s}{r} \right) \frac{s}{r} + \psi'' \left( \frac{s}{r} \right) \frac{s^2}{r^2} \right) \ds 
     \\
     & =  - \frac{|\theta(x_0)|^2\omega_n}{n} \left( \int_0^r  2 \psi' \left( \frac{s}{r} \right) \frac{s^n}{r} \ds + \left[ \psi'\left( \frac{s}{r} \right) \frac{s^{n+1}}{r} \right]_{s=0}^{s=r} - \int_0^r (n+1) \psi'\left( \frac{s}{r} \right) \frac{s^n}{r} \ds \right)
     \\ &
      =  - \frac{|\theta(x_0)|^2\omega_n}{n} \int_0^r (1-n)\psi'\left( \frac{s}{r} \right) \frac{s^n}{r} \ds = \frac{(n-1)\omega_n}{n}|\theta(x_0)|^2 r^n \int_0^1 \psi'(u) u^n \; \mathrm{d}u. 
\end{align}
Note that by the choice of $\psi$ one has
\begin{equation}
    I_n := \int_0^1 \psi'(u) u^n  \; \mathrm{d}u < 0.
\end{equation}
This, \eqref{eq:Qbere}, \eqref{eq:alphainfbere} together with \eqref{eq:maincompu} yield 
\begin{equation}
\frac{(n-1) \omega_n I_n}{n} |\theta(x_0)|^2 r^n 
+ o(1) r^n |\theta(x_0)|  = \int_{\Gamma \cap B_r(x_0)} Q \phi_r \; \mathrm{d}\mathcal{H}^{n-1}.  
\end{equation}
Taking absolute values  we infer 
\begin{equation}
\frac{(n-1) \omega_n I_n}{n} |\theta(x_0)|^2 r^n \leq o(1) r^n |(x_0)|+ ||Q||_{L^\infty} \int_{\Gamma \cap B_r(x_0)} |\phi_r| \; \mathrm{d}\mathcal{H}^{n-1}.
\end{equation}
Since $|\phi_r(x)| \leq |\theta(x_0)| r$ on $B_r(x_0)$ one has
\begin{equation}
\frac{(n-1) \omega_n I_n}{n} |\theta(x_0)|^2 r^n \leq o(1) r^n |\theta(x_0)| + ||Q||_{L^\infty} |\theta(x_0)| r \mathcal{H}^{n-1} ( \Gamma \cap B_r(x_0) )  
\end{equation}
Now dividing by $\frac{(n-1) \omega_n I_n}{n} r^n|\theta(x_0)| $ we obtain 
\begin{equation}
|\theta(x_0)|  \leq o(1)   +\frac{n}{(n-1) I_n} ||Q||_{\infty} \frac{\mathcal{H}^{n-1}( \Gamma \cap B_r(x_0)) }{\omega_{n}r^{n-1}}.
\end{equation}
We can now let $r \rightarrow 0 +$. Since $\Gamma = \partial \Omega' \in C^\infty$ we find that $\frac{\mathcal{H}^{n-1}( \Gamma \cap B_r(x_0)) }{\omega_{n}r^{n-1}} \rightarrow 1$, $(r \rightarrow 0 )$, (cf. \cite[Theorem 5.14 and 5.15]{EvGar}). It follows that 
\begin{equation}
|\theta(x_0)| \leq C(n) ||Q||_{L^\infty},
\end{equation}
as claimed. 
\end{proof}

\begin{cor}(A priori estimate) \label{cor:410}
Suppose that $\Gamma = \partial \Omega' \in C^\infty$ and $Q \in W^{2,p}(\Omega)$ for some $p > n$. Then the solution $u$ of \eqref{eq:measdir} lies in $C^{0,1}(\overline{\Omega})$ and satisfies
\begin{equation}
||\nabla u ||_{L^\infty ( \Omega ) }\leq C_1(n) ||Q||_{L^\infty(\Omega) } + C_2(\Omega, \mathrm{dist}(\Gamma, \partial \Omega)) ||u||_{L^\infty(\Omega)},
\end{equation} 
where $C_2$ depends increasingly on $\mathrm{dist}(\Gamma, \partial \Omega) + \frac{1}{\mathrm{dist(\Gamma, \partial \Omega)}}$.
Moreover, 
   \begin{equation}
    ||\nabla u ||_{L^\infty(\Omega)}  \leq C(\Omega, \mathrm{dist}(\Gamma, \partial \Omega), [\Gamma]_{Lip})||Q||_{L^\infty(\Omega)}, 
\end{equation}
where $C$ depends increasingly on $[\Gamma]_{Lip}$ and $\mathrm{dist}(\Gamma, \partial \Omega) + \frac{1}{\mathrm{dist(\Gamma, \partial \Omega)}}$
\end{cor}
\begin{proof}
The first estimate is a direct consequence of Corollary \ref{cor:BVreg}, \eqref{eq:cor48} and Lemma \ref{lem:49}. The second estimate follows from the first one, Lemma \ref{lem:310} and Lemma \ref{lem:divF}.
\end{proof}

\subsection{Completion of the proof by approximation}

\begin{proof}[Proof of Theorem \ref{thm:lipreg}]
Let $\Omega, \Omega', \Gamma,Q$ be as in the statement and $v$ be the unique solution to \eqref{eq:measdir}.\\
\textbf{Step 1:} Assume additionally that $\Gamma= \partial \Omega' \in C^\infty $  and $Q \in W^{2,p}(\Omega)$ for some $p > n$. Then the claim follows from Lemma \ref{lem:310} and  Corollary \ref{cor:410}. \\
\textbf{Step 2:} Assume additionally that $\Gamma = \partial \Omega' \in C^{\infty}$ and $Q \in C(\Gamma)$. By Tietze's extension theorem we can find $\widetilde{Q} \in C(\mathbb{R}^n)$ such that $\widetilde{Q} \vert_{\Gamma} = Q$. For fixed $\epsilon \in (0,1)$ let $\psi_\epsilon$ be the standard mollifier and note that $Q_\epsilon := \widetilde{Q} * \psi_\epsilon \vert_\Omega$ satisfies  $Q_\epsilon \in W^{2,\infty}(\Omega)$ and $Q_\epsilon$ converges to $Q$ uniformly on $\Gamma$.  Note also that for $\epsilon < 1$ one has $||Q_\epsilon||_{L^\infty(\Omega)} \leq ||\widetilde{Q}||_{L^\infty(B_1(\Omega))} < \infty$, where $B_1(\Omega):= \{x \in \mathbb{R}^n: \mathrm{dist}(x,\Omega) <1 \}$. Let $v_\epsilon \in L^1(\Omega)$ be the unique solution of 
\begin{equation}
\begin{cases}
    - \Delta v_\epsilon  = Q_\epsilon \vert_\Gamma \; \mathcal{H}^{n-1} \mres \Gamma  & \textrm{in } \Omega,  \\
   \quad \; \;  v_\epsilon = 0 & \textrm{on } \partial \Omega.
    \end{cases}
\end{equation}
By Step 1 and Corollary \ref{cor:410} we have $v_\epsilon \in C^{0,1}(\overline{\Omega})$ and
\begin{equation}
    ||\nabla v_\epsilon||_{L^\infty(\Omega)} \leq C(\Omega, \mathrm{dist}(\Gamma,\partial \Omega), [\Gamma]_{Lip}) ||Q_\epsilon||_{L^\infty(\Gamma)} \leq  C(\Omega, \mathrm{dist}(\Gamma,\partial \Omega), [\Gamma]_{Lip}) ||\widetilde{Q}||_{L^\infty(B_1(\Omega))}.
\end{equation}

Recalling that $v_\epsilon\vert_{\partial \Omega} = 0$ we find that  $(v_\epsilon)_{\epsilon \in (0,1)}$ defines a bounded family in $W^{1,\infty}(\Omega)$. Using Proposition \ref{prop:C01approx} we obtain a sequence $\epsilon_n \rightarrow 0$ and $u \in C^{0,1}(\overline{\Omega})$ such that $u \vert_{\partial \Omega} = 0$ and $v_{\epsilon_n} \rightarrow u$ uniformly on $\overline{\Omega}$. Moreover
\begin{equation}
    ||\nabla u||_{L^\infty(\Omega)} \leq \liminf_{n \rightarrow \infty} ||\nabla v_{\epsilon_n} ||_{L^\infty(\Omega)} \leq C(\Omega, \mathrm{dist}(\Gamma,\partial \Omega), [\Gamma]_{Lip})||\widetilde{Q}||_{L^\infty(B_1(\Omega))}.
\end{equation}
We claim next that $u = v$. From this follows the $C^{0,1}$-regularity  of $v$ by the previous inequality. We show that $u$ solves \eqref{eq:measdir} with the same data as $v$. To this end fix $\phi \in C^2(\overline{\Omega})$ such that $\phi \vert_{\partial \Omega} = 0$ and observe
\begin{equation}
    \int_\Omega u \Delta \phi \dx = \lim_{n \rightarrow \infty} \int_\Omega v_{\epsilon_n} \Delta \phi \dx = - \lim_{n \rightarrow \infty} \int_\Gamma Q_{\epsilon_n} \phi \; \mathrm{d}\mathcal{H}^{n-1}  = \int_\Gamma Q \phi \; \mathrm{d}\mathcal{H}^{n-1},
\end{equation}
where the last identity holds due to uniform convergence. By uniqueness of solutions of \eqref{eq:measdir}, $v = u$ and we obtain also
\begin{equation}\label{eq:Step2}
    ||\nabla v||_{L^\infty(\Omega)} \leq C(\Omega, \mathrm{dist}(\Gamma,\partial \Omega) , [\Gamma]_{Lip} ) || \widetilde{Q}||_{L^\infty(B_1(\Omega))}.
\end{equation}
\\
\textbf{Step 3.} Assume now that $\Gamma = \partial \Omega'\in C^{0,1}$ and $Q \in C(\Gamma)$. By Tietze's extension theorem we can again choose $\widetilde{Q} \in C(\mathbb{R}^n)$ such that $\widetilde{Q}\vert_{\Gamma} = Q.$ Next we can choose for $U = \Omega'$ approximating domains $(U_k)_{k \in \mathbb{N}}$ as in Proposition \ref{prop:dompert}. By Proposition \ref{prop:35} one has that 
$\mathrm{dist}(U_k,\partial \Omega) \rightarrow \mathrm{dist}(\Gamma, \partial \Omega) > 0 $ and hence there exists $\delta > 0$ such that $\mathrm{dist}(U_k, \partial \Omega) > \delta$. Moreover by the same proposition there exists $L> 0$ such that $[\partial U_k]_{Lip} < L$ for all $k$. Next we let $v_k \in L^1(\Omega)$ be the unique weak solution of 
\begin{equation}
    \begin{cases}
    - \Delta v_k = \widetilde{Q} \vert_{\partial U_k} \mathcal{H}^{n-1} \mres {\partial U_k} & \textrm{in } \Omega, \\
     \quad \; \; v_k =  0 & \textrm{on } \partial \Omega .
    \end{cases}
\end{equation}
We infer from Step 2 and \eqref{eq:Step2} that $v_k \in C^{0,1}(\overline{\Omega})$ and
\begin{equation}\label{eq:Step3mid}
    ||\nabla v_k||_{L^\infty(\Omega)} \leq C(\Omega, \mathrm{dist}(\partial U_k,\partial \Omega) , [\partial U_k]_{Lip} ) || \widetilde{Q}||_{L^\infty(B_1(\Omega))}.
\end{equation}
Since $\mathrm{dist}(\partial U_k, \partial \Omega) > \delta $ and $[\partial U_k]_{Lip} < L$ we infer that $\nabla v_k$ is uniformly bounded in $L^\infty( \Omega)$. Further, since $v_k \vert_{\partial \Omega} = 0$, we find that $v_k$ is uniformly bounded in $W^{1,\infty}(\Omega)$. By Proposition \ref{prop:C01approx} we may assume that (up to a subsequence, which we do not relabel) $v_k \rightarrow u$ uniformly on $\overline{\Omega}$ for some $u \in C^{0,1}(\overline{\Omega})$. We show next that $u = v$, i.e. $u$ solves \eqref{eq:measdir}. To this end we compute for $\phi \in C^2(\overline{\Omega})$ with $\phi \vert_{\partial \Omega} = 0$ that 
\begin{equation}
    \int_\Omega u \Delta \phi \; \mathrm{d}x = \lim_{k \rightarrow \infty} \int_{\Omega} v_k \Delta \phi \; \mathrm{d}x = \lim_{k \rightarrow \infty} \int_{\partial U_k} \widetilde{Q} \phi \; \mathrm{d}\mathcal{H}^{n-1} = \int_\Gamma \widetilde{Q} \phi \; \mathrm{d}\mathcal{H}^{n-1},
\end{equation}
where we used Proposition \ref{prop:35} in the last identity. By uniqueness we infer that $v = u$. Using that $\widetilde{Q} \vert_\Gamma = Q$, we obtain the claim. Note that by \eqref{eq:Step3mid} and the fact that $\mathrm{dist}(\partial U_k, \partial \Omega) \rightarrow \mathrm{dist}(\Gamma, \partial \Omega)$ and $[\partial U_k]_{Lip} \leq L$ we obtain
\begin{equation}
    ||\nabla v ||_{L^\infty(\Omega)} \leq C( \Omega , \mathrm{dist}( \Gamma, \partial \Omega) , L) || \widetilde{Q}||_{L^\infty(B_1(\Omega))},
\end{equation}
for some $L> 0$.
We remark finally that we may choose an extension $\widetilde{Q}$  of $Q$ such that $|| \widetilde{Q}||_{L^\infty(B_1(\Omega))} = ||Q||_{L^\infty( \Gamma)}$. Indeed, if $\widetilde{Q} \in C(\mathbb{R}^n)$ is an arbitary extension then $$\widetilde{\widetilde{Q}} := \max\{- ||Q||_{L^\infty(\Gamma)} , \min \{ \widetilde{Q}, ||Q||_{L^\infty(\Gamma)} \} \} \in C(\mathbb{R}^n)$$ does the job. We obtain that for some $L  > 0$
\begin{equation}\label{eq:Step3fin}
    ||\nabla v ||_{L^\infty(\Omega)} \leq C( \Omega , \mathrm{dist}( \Gamma, \partial \Omega) , L) ||Q||_{L^\infty(\Gamma)}.
\end{equation}
\\
\textbf{Step 4.} Assume the assumptions in the statement, i.e. $\Gamma$ has Lipschitz boundary and $Q \in L^\infty(\Gamma)$ is arbitrary. Let $L$ be as in Step 3. 
By Proposition \ref{prop:L^infapp} there exists a sequence $(Q_k)_{k \in \mathbb{N}} \subset C(\Gamma)$ such that $||Q_k||_{L^\infty(\Gamma)} \leq ||Q||_{L^\infty(\Gamma)}$ and $Q_k \rightarrow Q $ in $L^1(\Gamma)$. Again let $v_k \in L^1(\Omega)$ be the weak solutions to 
\begin{equation}
    \begin{cases}
    - \Delta v_k = Q_k \;  \mathcal{H}^{n-1}\mres \Gamma  & \textrm{in } \Omega \\ \quad \; \;  v_k = 0 & \textrm{on } \partial \Omega 
    \end{cases}
\end{equation}
By \eqref{eq:Step3fin}  and   $v_k\vert_{\partial \Omega} = 0$ we infer that $(v_k)_{k \in \mathbb{N}}$ is uniformly bounded in $W^{1,\infty}(\Omega)$ and 
    \begin{equation}
    ||\nabla v_k ||_{L^\infty(\Omega)} \leq C( \Omega , \mathrm{dist}( \Gamma, \partial \Omega) , L) ||Q_k||_{L^\infty(\Gamma)} \leq C( \Omega , \mathrm{dist}( \Gamma, \partial \Omega) , L) ||Q||_{L^\infty(\Gamma)}.
\end{equation}
By Proposition \ref{prop:C01approx} we infer that a subsequence (which we do not relabel) has a uniform limit $u \in C^{0,1}(\overline{\Omega})$ satisfying
\begin{equation}
    ||\nabla u ||_{L^\infty(\Omega)} \leq C( \Omega , \mathrm{dist}( \Gamma, \partial \Omega) , L) ||Q||_{L^\infty(\Gamma)}.
\end{equation}
We show next that $u =v$ by showing again that $u$ solves $\eqref{eq:measdir}$. Indeed one has (since $Q_k \rightarrow Q$ in $L^1(\Gamma)$) for each $\phi \in C^2(\overline{\Omega})$ such that $\phi\vert_{\partial \Omega} = 0$
 \begin{equation}
     \int_\Omega u \Delta \phi \dx = \lim_{k \rightarrow \infty} \int_\Omega v_k \Delta \phi \dx = \lim_{k \rightarrow \infty} - \int_\Gamma Q_k \phi \; \mathrm{d}\mathcal{H}^{n-1} = - \int_\Gamma Q \phi \; \mathrm{d}\mathcal{H}^{n-1}.
 \end{equation}
 The claim follows.
\end{proof}

\section{Proof of Main Theorem \ref{thm:lipgraphmain}}

We have now shown the Lipschitz regularity for \emph{closed} Lipschitz manifolds. Next we want to discuss the case of general compact Lipschitz manifolds --- with or without boundary. To this end we first look at Lipschitz graphs. 

\begin{theorem}[Lipschitz regularity for Lipschitz graphs]\label{thm:lipgraph}
Suppose that $\Omega \subset \mathbb{R}^n$ is a bounded domain with smooth boundary. Let $\Gamma = \{ (y, f(y)) : y \in U \} \subset \subset \Omega$ be a graph of a Lipschitz function $f: U \rightarrow \mathbb{R}$, where $U \subset  \mathbb{R}^{n-1}$ is such that $\partial U$ is a set of vanishing $n-1$ dimensional Lebesgue measure. Further let $Q \in L^\infty(\Gamma)$. Then the unique solution of \eqref{eq:measdir} lies in $C^{0,1}(\overline{\Omega})$.
\end{theorem}

\begin{proof}
Let $\Gamma,Q$ be as in the statement and $v \in L^1(\Omega)$ be the unique solution of \eqref{eq:measdir}.
First note that $f$ is uniformly continuous on $U$ and hence extends uniquely to a Lipschitz continuous function on $\overline{U}$. We divide the proof in two steps.\\
\textbf{Step 1}: We assume additionally that there exists $\epsilon >0$ such that $R := U  \times (- \inf_U f - \epsilon, \sup_U f + \epsilon) \subset \subset \Omega$.  We define 
$
    \Omega' := \{(x,y) : (x,y) \in R , y > f(x) \} . 
$
Note that $\Omega' \subset \subset \Omega$ is a domain with Lipschitz boundary, cf. Proposition \ref{prop:Lipschitzsub}. Moreover the boundary is cointained in $\Gamma \cup \partial R$. Define $\widetilde{Q} : \Gamma \cup \partial R \rightarrow \mathbb{R}$ via
\begin{equation}
    \widetilde{Q}(x) := \begin{cases} Q(x) & x \in \Gamma, \\ 0 &  x \in \partial R .
    \end{cases} 
\end{equation}
This is well-defined as $\partial R \cap \Gamma = \emptyset$   and we have $ \widetilde{Q} \in L^\infty(\Omega)$.
Next let $u \in L^1(\Omega)$ be the unique weak solution of 
\begin{equation}
    \begin{cases}
    -\Delta u = \widetilde{Q} \;  \mathcal{H}^{n-1}\mres \partial \Omega'  & \textrm{in }  \Omega, \\
   \quad \; \;  u = 0 & \textrm{on } \partial \Omega.
    \end{cases}
\end{equation}
We infer from Main Theorem \ref{thm:lipreg} that $u \in C^{0,1}(\overline{\Omega})$. Next we show that $v = u$. To this end, we show that $u$ solves \eqref{eq:measdir}. Indeed, let $\phi \in C^2(\overline{\Omega})$ be such that $\phi\vert_{\partial \Omega} = 0$. Then by the choice of $\widetilde{Q}$ we have
\begin{equation}
    \int_\Omega u \Delta \phi \dx = \int_{\partial \Omega'} \widetilde{Q} \phi \; \mathrm{d}\mathcal{H}^{n-1} = \int_\Gamma Q \phi \; \mathrm{d}\mathcal{H}^{n-1}. 
\end{equation}
By uniqueness we infer that $u = v$ and thus $v \in C^{0,1}(\overline{\Omega})$.
\\
\textbf{Step 2:} Let $\Gamma$ be as in the statement without any further assumption. We claim that 
$\Gamma$ can be written as union of finitely many disjoint Lipschitz graphs $(\Gamma_i)_{i =1}^r$ all of which satisfy the additional assumption of Step 1. To this end let $p \in \overline{U}$ be arbitrary. Then $\mathrm{dist}((p, f(p)), \partial \Omega) > 0$. Define $\epsilon(p):= \frac{1}{2\sqrt{5}} \mathrm{dist}((p, f(p)), \partial \Omega)$ and choose $\delta(p) \in (0 , \epsilon(p))$ such that for all $q \in \overline{U} \cap B_{\delta(p)}(p)$ one has $|f(p)-f(q)| < \epsilon(p)$. We claim that $R(p) := (U \cap B_{\delta(p)}(p)) \times (\inf_{B_{\delta(p)}(p)} f - \epsilon(p), \sup_{B_{\delta(p)}(p)} f + \epsilon(p)) \subset \subset \Omega$. Indeed, if $z =(z', z_n) \in R(p)$ we can find $q \in B_{\delta(p)} (p) $ such that $|z_n - f(q)| < \epsilon (p)$. Now
\begin{equation}
    |z-(p,f(p))|^2 \leq \delta(p)^2 + |z_n- f(p)|^2 < \epsilon(p)^2 + (|z_n - f(q)| + |f(q)- f(p)|)^2 \leq 5 \epsilon(p)^2.
\end{equation}
The choice of $\epsilon(p)$ implies that $|z- (p,f(p))| < \frac{1}{2} \mathrm{dist}((p,f(p)), \partial \Omega)$ for all $z \in R(p)$ and hence we infer that $R(p) \subset \subset \Omega$. Note that $\overline{U} \subset \bigcup_{p \in \overline{U}} B_{\delta(p)}(p)$ and hence we may select a finite family $(p_i)_{i = 1}^r$ such that $\overline{U} =  \bigcup_{i = 1}^r ( \overline{U} \cap B_{\delta(p_i)} (p_i))$. For $i = 1,..., r$ define 
\begin{equation}
    \Gamma_i := \left\lbrace (x, f(x)) : x \in U \cap B_{\delta(p_i)}(p_i) \setminus \bigcup_{j = 1}^{i-1} ( \overline{U \cap B_{\delta(p_j)} (p_j))} \right\rbrace,
\end{equation}
\begin{equation}
    R_i := \left( U \cap B_{\delta(p_i)}(p_i) \setminus \bigcup_{j = 1}^{i-1} ( \overline{U \cap {B}_{\delta(p_j)} (p_j)}) \right) \times (\inf_{B_{\delta(p_i)}(p_i)} f - \epsilon(p_i) + \sup_{B_{\delta(p_i)}(p_i)} f + \epsilon(p_i) ) .
\end{equation}
We infer that $\Gamma \cap R_i = \Gamma_i$ 
and $\Gamma$ is up to a set of $\mathcal{H}^{n-1}$-measure zero the disjoint union of $(\Gamma_i)_{i = 1}^r$. Indeed, each $z \in \Gamma \setminus \bigcup_{i = 1}^r \Gamma_i$ must be given by $z = (y,f(y))$ for some $y \in \bigcup_{j=1}^r \partial ( U \cap B_{\delta(p_j)}(p_j))$. Now observe that by the area formula (cf. \cite[Section 3.3.2]{EvGar})
\begin{equation}\label{eq:Husmeas}
    \mathcal{H}^{n-1} \left( \left\lbrace (y, f(y)) : y \in \partial (U \cap B_{\delta(p_j)}(p_j)) \right\rbrace \right)   = 0 \quad \forall j \in \mathbb{N}.  
\end{equation}
Here we needed that $\partial U$ is an $n-1$ dimensional null set since  then $\partial (U \cap B_{\delta(p_j)}(p_j)) \subset \partial U \cup \partial B_{\delta(p_j)}(p_j)$ is also an $n-1$ dimensional null set.
Due to the fact that $\Gamma_i \subset R_i \subset \subset \Omega$ we conclude that $\Gamma_i$ satisfies the assumptions of Step 1. Since 
\begin{equation}
    \mathcal{H}^{n-1} \mres \Gamma = \sum_{i = 1}^r \mathcal{H}^{n-1} \mres \Gamma_i
\end{equation}
we conclude by the superposition principle and the uniqueness of the solution that 
\begin{equation}
    v = \sum_{ i = 1} v_i,
\end{equation}
where $v_i \in L^1(\Omega)$ is the unique solution of 
\begin{equation}
    \begin{cases}
    - \Delta v_i = Q \mathcal{H}^{n-1} \mres \Gamma_i  & \textrm{in }\Omega, \\
    \quad \; \; v_i = 0 & \textrm{on } \partial \Omega.
     \end{cases}
\end{equation}
By Step 1 $v_i \in C^{0,1}(\overline{\Omega})$ for all $i$ and hence we conclude that $v \in C^{0,1}(\overline{\Omega}).$
\end{proof}

\begin{remark}\label{rem:egal}
In the statement of the previous theorem we needed to require the domain of definition $U$ of $f$ satisfies that $\partial U$ is a set of $n-1$ dimensional Lebesgue measure zero. This in particular implies (cf. \cite[Section 3.3.2]{EvGar})
\begin{equation}
    \mathcal{H}^{n-1} ( \{ (y,f(y)) : y \in \partial U \} ) = 0 ,
\end{equation}
which in turn yields that the set $\Gamma$ coincides $\mathcal{H}^{n-1}$ a.e. with the (compact) Lipschitz manifold $\{ (y, f(y)): y \in \overline{U} \}$. Hence the previous theorem can already be seen as a result about compact Lipschitz manifolds.
\end{remark}

\begin{remark}\label{rem:partitionofunit}
Instead of the decomposition of $\Gamma$ in Step 2 of the proof of Theorem \ref{thm:lipgraph} we could as well argue with a partition of unity on $\Gamma$. Such partition of unity exists since $\Gamma$ can be understood as a smooth (!) manifold with a smooth structure generated by the chart $\psi : \Gamma \rightarrow \mathbb{R}^{n-1}$, $\psi(x,f(x)) := x, (x \in U)$. We decided here for the  decomposition approach since it is more elementary. 
\end{remark}

Having shown the regularity for Lipschitz graphs it is easy to obtain the desired result for compact Lipschitz manifolds --- they can simply be covered by finitely many Lipschitz graphs.

\begin{proof}[Proof of Main Theorem \ref{thm:lipgraphmain}]
Let $v \in L^1(\Omega)$ be the unique solution of \eqref{eq:measdir} and $\Gamma,Q$ be as in the statement. 
Since $\Gamma$ is a compact Lipschitz manifold we can find $N \in \mathbb{N}$, $U_1,..., U_N \subset \mathbb{R}^{n-1}$ open or half-open rectangles (cf. Appendix \ref{app:lipschitzDom}), $V_1,...,V_N \subset \mathbb{R}$ open and bounded, $R_1,...,R_N \in O(n)$ and $f_i : U_i \rightarrow V_i$ Lipschitz functions such that 
\begin{equation}
    \Gamma \cap R_i (U_i \times V_i) = R_i \{ (x, f_i(x)) : x \in U_i \}.
\end{equation}
Note that all those Lipschitz functions $f_i: U_i \rightarrow V_i$ also have a Lipschitz  extension to the closure $\overline{U}_i$, which we will also call $f_i$. 
By possibly shrinking the sets $U_i$ we can also achieve that 
\begin{equation}\label{eq:bars}
    \Gamma \cap R_i (\overline{U_i} \times \overline{V_i}) = R_i \{ (x, f_i(x)) : x \in \overline{U_i} \} \quad \forall i = 1,...,N.
\end{equation}
We define 
$
    \Gamma_i := \Gamma \cap R_i (U_i \times V_i) \setminus \bigcup_{j = 1}^{i-1} \overline{\Gamma \cap R_j (U_j \times V_j)}.
$
We claim that up to a set of $\mathcal{H}^{n-1}$- measure zero $\Gamma$ is the disjoint union of $(\Gamma_i)_{i = 1}^N$. That $(\Gamma_i)_{i = 1}^N$ is a pairwise disjoint family is clear from the construction. To see that the union of $(\Gamma_i)_{i = 1}^N$ coincides up to a set of Hausdorff measure zero with $\Gamma$ it suffices to show that 
\begin{equation}
    \mathcal{H}^{n-1} \left( \overline{\Gamma \cap R_i(U_i \times V_i) } \setminus ( \Gamma \cap R_i(U_i \times V_i) \right) = 0 \quad \forall i = 1,...,N.
\end{equation}
This becomes obvious when observing by rotational invariance of the Hausdorff measure and the area formula  
\begin{align}
 \mathcal{H}^{n-1} \left( \overline{\Gamma \cap R_i(U_i \times V_i) } \setminus ( \Gamma \cap R_i(U_i \times V_i) \right) & \leq \mathcal{H}^{n-1}( \Gamma \cap R_i (  \overline{U}_i \setminus U_i \times \mathbb{R} ) )
 \\ & \leq \mathcal{H}^{n-1}( R_i \{(x, f_i(x)): x \in \overline{U}_i \setminus U_i \} ) = 0.
 \end{align} 
 By the superposition principle we conclude that 
 \begin{equation}\label{eq:superposi}
     v = \sum_{i = 1}^N v_i ,
 \end{equation}
where $v_i \in L^1(\Omega)$ is the unique solution to 
\begin{equation}
\begin{cases}
    -\Delta v_i = Q \mathcal{H}^{n-1}\mres \Gamma_i & \textrm{in } \Omega, \\
   \quad \; \;  v_i = 0 & \textrm{on } \partial \Omega. \end{cases} \quad \quad  (i = 1,...,N)
\end{equation}
We show next that $v_i \in C^{0,1}( \overline{\Omega})$ for all $i$. This and \eqref{eq:superposi} will then imply the claim. To show the desired Lipschitz regularity we apply Theorem \ref{thm:lipgraph}. To this end define for all $i = 1,...,N$ the set $K_i := R_i^{-1}\left( \bigcup_{j = 1}^{i-1} \overline{\Gamma \cap R_j ( U_j \times V_j)} \right)$. Observe that $K_i \subset \Gamma$ is compact and $C_i := \{ x' \in \overline{U_i} : \exists z \in \overline{V_i} \; s.t. \;  (x',z) \in R_i^{-1}(\Gamma) \cap K_i \}$ is also compact. We rewrite 
\begin{align}
    \Gamma_i & = \Gamma \cap R_i (U_i \times V_i)  \setminus R_i K_i 
    = R_i \{ ( x, f_i(x)) : x \in U_i \} \setminus R_i K_i 
    \\ &= R_i \left( \{ ( x, f_i(x)) : x \in U_i \} \setminus K_i \right) 
   = R_i \left( \{ (x, f_i(x) ) : x \in U_i \setminus C_i \} \right).\label{eq:Gammai}
\end{align}
This representation brings us closer to applicability of Theorem \ref{thm:lipgraph} - we only need to discuss the rotation $R_i$ and the domain of definition $U_i \setminus C_i$. 
One  needs to show that  $\partial (U_i \setminus C_i)  $ has vanishing $n-1$ dimensional Lebesgue measure.
To this end observe that 
$U_i = (id, f_i)^{-1}( R_i^{-1}(\Gamma) \cap (U_i \times V_i))$ and by \eqref{eq:bars} one has $C_i = (id, f_i)^{-1}(R_i^{-1}(\Gamma) \cap K_i \cap ( \overline{U}_i \times \overline{V}_i))$. This yields 
\begin{equation}
    (id, f_i)(U_i \setminus C_i) = R_i^{-1}(\Gamma \cap R_i(U_i \times V_i) ) \setminus K_i = R_i^{-1} \left(\Gamma \cap (U_i \times V_i) \setminus \bigcup_{j =1 }^{i-1} \overline{\Gamma \cap R_j(U_j \times V_j)}   \right). 
\end{equation}
Now if $\Pi: \mathbb{R}^n \rightarrow \mathbb{R}^{n-1}$ denotes the projection on the first $n-1$ coordinates one has 
\begin{align}
    U_i \setminus C_i & = \Pi \left( R_i^{-1} \left(\Gamma \cap R_i(U_i \times V_i) \setminus \bigcup_{j =1 }^{i-1} \overline{\Gamma \cap R_j(U_j \times V_j)}   \right) \right),\\
   \overline{U_i \setminus C_i} & = \Pi \left( R_i^{-1} \left(\overline{\Gamma \cap R_i(U_i \times V_i)} \setminus \bigcup_{j =1 }^{i-1} \overline{\Gamma \cap R_j(U_j \times V_j)}   \right) \right).
\end{align}
This yields immediately that 
\begin{equation}
    |\partial (U_i \setminus C_i) | = |\overline{U_i \setminus C_i} \setminus (U_i \setminus C_i)| \leq \sum_{j = 1}^{i} \mathcal{H}^{n-1} ( \overline{\Gamma \cap R_j(U_j \times V_j)} \setminus \Gamma \cap R_j(U_j \times V_j) ) = 0.
\end{equation}
We next get rid of the rotation $R_i$ by using the rotation invariance of the Laplacian. To this end claim that $w_i := v_i( R_i (\cdot)) $ solves 
\begin{equation}\label{eq:lapwi}
    \begin{cases}
    - \Delta w_i = (Q \circ R_i) \mathcal{H}^{n-1} \mres \{ (x, f_i(x)) : x \in U_i \setminus C_i \} & \textrm{in } R_i^{-1}\Omega, \\
    \quad \; \; w_i = 0 & \textrm{on } \partial (R_i^{-1} \Omega).
    \end{cases}
\end{equation}
Once this is shown it follows from Theorem \ref{thm:lipgraph}
 and Remark \ref{rem:egal} that $w_i \in C^{0,1}( \overline{ R_i \Omega})$ which in turn implies that $v_i = w_i ( R_i^{-1}(\cdot)) $ lies in $C^{0,1}(\overline{\Omega})$. It only remains to show \eqref{eq:lapwi}. To this end let $\phi \in C^2(\overline{R_i^{-1} \Omega})$ be  such that $\phi_{\vert_{\partial (R_i^{-1} \Omega)}} = 0$. Using that $|\mathrm{det}(R_i)| = 1$, the rotational invariance of $\Delta$ and of $\mathcal{H}^{n-1}$ and \eqref{eq:Gammai}  we obtain
 \begin{align}
     \int_{R_i^{-1} \Omega} w_i(x) \Delta \phi(x) \dx & =  \int_{R_i^{-1} \Omega } v_i(R_i x) \Delta \phi (x) \dx = \int_\Omega v_i(y) \Delta \phi (R_i y) \dy 
     \\ & = \int_{\Omega} v_i(y) \Delta ( \phi \circ R_i) (y) \dy = \int_{\Gamma_i} Q(z) \phi(R_i(z)) \; \mathrm{d}\mathcal{H}^{n-1}(z)
     \\ & = \int_{R_i\left( \{ (x, f_i(x) ) : x \in U_i \setminus C_i \} \right)} Q(z) \phi(R_i(z)) \; \mathrm{d}\mathcal{H}^{n-1}(z) 
     \\ & = \int_{\left( \{ (x, f_i(x) ) : x \in U_i \setminus C_i \} \right)} (Q\circ R_i)(s) \phi(s) \; \mathrm{d}\mathcal{H}^{n-1}(z) .
 \end{align}
 This shows \eqref{eq:lapwi} and the claim follows then  as discussed above from Theorem \ref{thm:lipgraph}, Remark \ref{rem:egal} and \eqref{eq:superposi}.
 \end{proof}

\section{Application: Regularity for the biharmonic Alt-Caffarelli Problem}\label{sec:appli}

\subsection{Description of the problem}

In this section we apply our findings to the biharmonic Alt-Caffarelli Problem in two dimensions, which we will introduce now. Suppose that $\Omega \subset \mathbb{R}^2$ is a smooth domain and $u_0 \in C^\infty( \overline{\Omega}), u_0 > 0$. We define
\begin{equation}
    \mathcal{A}(u_0) := \{ u \in W^{2,2}(\Omega) : u - u_0 \in W_0^{1,2}(\Omega)\} 
\end{equation}
and the energy functional $\mathcal{E}: \mathcal{A}(u_0) \rightarrow \mathbb{R}$ given by 
\begin{equation}
    \mathcal{E}(u) := \int_\Omega ( \Delta u )^2 \dx + |\{ u > 0 \}| \quad ( u \in \mathcal{A}(u_0)) , 
\end{equation}
where $| \cdot |$ denotes the Lebesgue measure. The energy balances out two conflicting interests: Minimizers must be nonpositive on a large set but at the same time have a small bending energy, for details see \cite{Marius}. Minimizing $\mathcal{E}$ in $\mathcal{A}(u_0)$ and the study of properties of minimizers has recently raised a lot of interest, cf. \cite{Dipierro2}, \cite{Dipierro1}, \cite{Marius}.

By \cite[Lemma 2.1]{Dipierro2} there exists $w \in \mathcal{A}(u_0)$ such that 
\begin{equation}\label{eq:minialt}
    \mathcal{E}(w)= \inf_{u \in \mathcal{A}(u_0)} \mathcal{E}(u).
\end{equation}
In \cite[Theorem 1.4]{Marius} it has been shown that $w \in C^2(\Omega)$, $\Omega' := \{w < 0 \} \subset \subset \Omega$ has $C^2$-boundary, given by $\Gamma := \{ w = 0\}$, and $\nabla w \neq 0$ on $\Gamma$. Moreover
$v := -\Delta w$ solves 
\begin{equation}
    \int v \Delta \phi \dx  =  - \int_{\Gamma} \frac{\phi}{2|\nabla w|} \; \mathrm{d}\mathcal{H}^1 \quad \forall \phi \in C^2(\overline{\Omega}):\;  \phi \vert_{\partial \Omega} = 0.   
\end{equation}
In our language, $v$ is a weak solution of 
\begin{equation}\label{eq:lapsol}
    \begin{cases}
    - \Delta v = Q \;  \mathcal{H}^1 \mres \Gamma & \textrm{in} \; \Omega, \\
  \quad \; \;   v = 0 & \textrm{on} \; \partial \Omega,
    \end{cases}
\end{equation}
where $Q(x) := \frac{1}{2|\nabla w(x)|}$ for all $x \in \Gamma$. 
\subsection{Optimal regularity}
\begin{prop}(Optimal regularity for $v = \Delta w$)
Suppose that $w \in \mathcal{A}(u_0)$ is a solution of \eqref{eq:minialt}. Then $v= \Delta w$ lies in $C^{0,1}(\overline{\Omega})$ and $\nabla v \in BV(\Omega)$
\end{prop}
\begin{proof}
Let $w,v$ be as in the statement and let $Q, \Omega', \Gamma$ be defined as in the beginning of this section. 
Since by \cite[Theorem 1.4]{Marius} $\nabla w \in C^1(\Omega)$ and $\nabla w \neq 0$ on $\Gamma \subset \subset \Omega$ we have that $Q = \frac{1}{2|\nabla w| } \in L^\infty(\Gamma)$. By \eqref{eq:lapsol} and Theorem \ref{thm:lipreg} we obtain that $v \in C^{0,1}(\overline{\Omega})$. Hence also $\nabla v \in L^\infty(\Omega)$. In order to obtain that $\nabla v \in BV(\Omega)$ we intend to use Corollary \ref{cor:BVreg}. To this end we have to show that $Q = \frac{1}{2|\nabla w|}$ extends to a function in $W^{2,p}(\Omega)$ for some $p > 2$. Since $\nabla w \in C^1$ and $\nabla w \neq 0$ on $\Gamma$  there exists an open neighborhood $U \subset \subset \Omega$ of $\Gamma$ and $\delta > 0$ that $|\nabla w|> \delta$ on $U$. Next we choose $\eta \in C_0^\infty(U)$ auch that $\eta \equiv 1$ on $\Gamma$ and define 
\begin{equation}
    \widetilde{Q}(x) := \begin{cases}  \frac{1}{|\nabla w(x)|}\eta (x) & x \in U, \\ 0 & x \in \Omega \setminus U. \end{cases}
\end{equation}
Obviously, $\widetilde{Q}$ is an extension of $Q$. We show next that $\widetilde{Q} \in W^{2,p}(\Omega)$ for some $p > 2$. Since $U \subset \subset \Omega$ it is sufficient to show by the chain rule that $\partial_i w \in W^{2,p}_{loc}(\Omega)$ for all $i = 1,...,n$. To this end we compute for all $\phi \in C_0^\infty(\Omega)$
\begin{equation}
    \int_\Omega \nabla \partial_i w \nabla \phi \dx  = \int_{\Omega} w \partial_i \Delta \phi \dx  = \int_{\Omega} w \Delta \partial_i \phi \dx  = \int_\Omega \Delta w \partial_i \phi \dx  = -\int_\Omega \partial_i v \phi \dx.
\end{equation}
This makes $\partial_i w$ a weak solution of $ \Delta (\partial_i w) = \partial_i v \in L^\infty(\Omega)$. Standard regularity theory implies that $ \partial_i w \in W^{2,p}_{loc}(\Omega)$ for all $p \in (1,\infty)$. Choosing any $p \in (2,\infty)$ we obtain the claim and the desired BV-regularity follows immediately by Corollary \ref{cor:BVreg}.
\end{proof}
This implies that $w \in C^2(\Omega)$ is a (classical) solution of 
\begin{equation}
    \begin{cases}
    - \Delta w= -v \in W^{1,\infty}(\Omega) & \textrm{in } \Omega \\
    \quad \; \; w = u_0 \in C^\infty( \overline{\Omega}) &  \textrm{on } \partial \Omega
    \end{cases}
\end{equation}
By elliptic regularity of the Laplacian (cf. \cite[Theorem 8.13]{GilTru}) one infers that $w \in W^{3,p}(\Omega)$ for all $p \in [1,\infty)$. Note that for $p = \infty$ it is not possible to conclude anything with this elliptic regularity argument. The question whether $w \in W^{3,\infty}(\Omega)$ is still unanswered. However at least $v = \Delta w$ lies in $W^{1,\infty}(\Omega)$ which is optimal. Whether further (variational or potential theoretic) arguments for the $W^{3,\infty}$-regularity can be found will be subject to future research. 


\appendix

\section{The (signed) distance function}\label{app:sign}
In this section we will define the signed distance function and discuss some of its basic properties. 
For a set $A \subset \mathbb{R}^n$ we define the \emph{$\epsilon$-parallel set} $A_\epsilon := \{ x \in \mathbb{R}^n : \mathrm{dist}(x,A) < \epsilon \}$. 

\begin{definition}
For $U \subset \mathbb{R}^n$ open and bounded we define 
\begin{equation}
d_U(x) := \begin{cases} 
\mathrm{dist}(x,U)  & x \in U^C,  \\ -\mathrm{dist}(x,U^C)  & x  \in U .
\end{cases}
\end{equation}
\end{definition}

\begin{lemma}[{cf. \cite[Appendix 14.6]{GilTru} }]\label{lem:signdist}
Let $U \subset \mathbb{R}^n$ be bounded with $C^k$-boundary, $k \geq 2$ and $S:= \partial U$. Then there exists $\epsilon_0 > 0$ such that  for all $\epsilon \in (0,\epsilon_0)$
\begin{itemize}
\item $d_U \in C^k( \overline{S_\epsilon})$,
\item For each $x \in S_\epsilon$ there exists a unique $\pi(x) \in S$ such that $|x-\pi(x)| = d_U(x)$, 
\item For all $x \in S_\epsilon$ one has $\nabla d_U(x) = \nu(\pi(x))$ where $\nu(z)$ denotes the outer unit normal of $U$ at $z$. 
\item For each $x \in S_\epsilon$ there exists an orthogonal matrix $S(x)$ such that $S(x)^TD^2d_U(x) S(x) = \mathrm{diag} \left(\frac{-\kappa_1(\pi(x))}{1- \kappa_1(\pi(x))d_U(x)},...\frac{-\kappa_{n-1}(\pi(x))}{1- \kappa_{n-1}(\pi(x))d_U(x)} \right)$, where $\kappa_1(z),..., \kappa_{n-1}(z)$ denote the principal curvatures of $S$ at $z$. In particular $\Delta d_U(x) = \sum_{i = 1}^{n-1} \frac{-\kappa_i(\pi(x))}{1- \kappa_i(\pi(x))d_U(x)}$ for all $x\in S_\epsilon$. 
\end{itemize}
\end{lemma}
Observe also that $|d_U(x)| = \mathrm{dist}(x,\partial U)$ is Lipschitz continuous, which is often used throughout the article.

\section{Lipschitz Manifolds and Lipschitz domains}\label{app:lipschitzDom}
In this section we recall properties of Lipschitz domains that are useful for our purposes. Recall that a bounded domain $D \subset \mathbb{R}^n$ is a \emph{Lipschitz domain} if for all $x \in \partial D$ there exists $U \subset \mathbb{R}^{n-1}$ open and $V \subset \mathbb{R}$ open as well as a rotation matrix $R \in O(n)$  and a Lipschitz function $f \in C^{0,1}(W)$ such that $x \in R(U\times V)$ and 
\begin{equation}
    D \cap R(U \times V) = R\{ (\widetilde{x},y) : x \in U, y < f(\widetilde{x}) \} . 
\end{equation}
One may as well assume that $R$ is a rigid euclidean motion instead of just a rotation matrix. Since $U$ is not required to be connected one may require that $V= (-a,\infty)$ for some $a > 0$. 
In any case $\partial D$ is a \emph{Lipschitz manifold (without boundary)}, i.e. represented locally by the graph of  Lipschitz functions.

For a compact Lipschitz manifold $\Gamma$ (without boundary) we can look at the \emph{Lipschitz constant} of $\Gamma$ given by
\begin{align}\label{eq:GammaLipschitzNorm}
    [\Gamma]_{Lip} := \inf \Big\lbrace \sum_{i = 1}^N & \sqrt{1+ ||\nabla f_i||_{L^\infty(U_i)}^2} : \;  N \in \mathbb{N}, U_1,...,U_N \subset \mathbb{R}^{n-1} \; V_1,...,V_n \subset \mathbb{R} \textrm{open}, \\ &  R_1,...,R_N \in O(n) \textrm{ and }  f_1: U_1 \rightarrow V_1,....f_N: U_N \rightarrow V_N \textrm{ Lipschitz s.t.} \\ &  \Gamma \subset \bigcup_{i = 1}^N  R_i(U_i \times V_i) \textrm{ and } ( \Gamma \cap R_i(U_i \times V_i))  = R_i \{(x,f_i(x)): x \in U_i\}  \Big\rbrace.
\end{align}
The Lipschitz constant will appear at several occasions in this discussion, which is why we have a notation for it.  
\begin{remark}
A shorter way to express $[\Gamma]_{Lip}$ is 
\begin{equation}
    [\Gamma]_{Lip} = \sup_{\phi \in L^1(\Gamma), ||\phi||_{L^1} = 1} \int_\Gamma \phi \; \mathrm{d}\mathcal{H}^{n-1}.
\end{equation}
We will not need this formula in this article which is why we do not give a proof. It may however be useful in several discussions.
\end{remark}

We say that $\Gamma \subset \subset \mathbb{R}^n$ is a \emph{compact Lipschitz manifold with boundary} if $\Gamma$ is (up to rotation)  given by a union of finitely many graphs of Lipschitz functions $f: U \rightarrow \mathbb{R}$ that are either defined on an open rectangle $U \subset \mathbb{R}^{n-1}$ or on a set of the form $U = (a_1,b_1) \times ... \times (a_{n-2}, b_{n-2}) \times (a_{n-1}, b_{n-1}] $.

One can define $[\Gamma]_{Lip}$ accordingly also for manifolds with boundary.

Next we will discuss properties of Lipschitz manifolds that we use throughout the article.

Another very important property for us is the approximation of Lipschitz domains with smooth domains. 

\begin{prop}[{Domain perturbation, cf. \cite[Theorem 8.3.1]{Daners} and \cite[Theorem 4.1 and 5.1]{Doktor}}]\label{prop:dompert}
Let $U \subset \mathbb{R}^n$ be a bounded Lipschitz domain. Then there exists a sequence $(U_k)_{k \in \mathbb{N}}$ of domains with $C^\infty$-boundary such that 
\begin{itemize}
\item $U \subset  U_k$ for all $k \in \mathbb{N}$, 
\item $\lim_{k \rightarrow \infty} | U_k \setminus U | = 0$,
\item for all $x \in \partial U$ there exists a  Euclidean motion $R_x$, an open neighborhood $B_x \subset \mathbb{R}^n$, a parameter $a_x > 0$ and a Lipschitz function $\psi^{(x)} : B_x \rightarrow \mathbb{R}$ such that $\psi^{(x)}(0)= 0$ that satisfies 
\begin{equation}
 U \cap R_x(B_x \times (-a_x , \infty) ) =R_x \{ (y', y_n) \in B_x \times (-a_x , \infty)  : y_n < \psi^{(x)} (y') \} 
\end{equation}
and for all $k \in \mathbb{N}$ there exists $\psi_k^{(x)} \in C^\infty(B_x))$ such that $\psi_k^{(x)} \geq \psi^{(x)}$ and 
\begin{equation}
 U_k \cap R_x(B_x \times (-a_x , \infty) )  =  R_x\{(y', y_n) \in B_x \times (-a_x , \infty)  : y_n < \psi_k^{(x)}(y') \},
\end{equation}
\item For all $x \in \partial U$ one has $\psi_k^{(x)} \rightarrow \psi^{(x)}$ uniformly in $B_x$,
\item There exists $M > 0$ such that for all $x \in \partial U$ one has $||\nabla \psi_k^{(x)}||_{L^\infty(B_x)} \leq M$ for all $k \in \mathbb{N}$,
\item For all $x \in \partial U$ and all $p \in [1, \infty)$ one has $\nabla \psi_k^{(x)} \rightarrow \nabla \psi^{(x)}$ in $L^p(B_x)$,
\item For each finite number $m \in \mathbb{N}$ and $x_1,...,x_m \in \partial U$ such that $\partial U \subset \bigcup_{j = 1}^m R_{x_j} (B_{x_j} \times (-a_{x_j}, \infty))$ one has $\partial U_k \subset \bigcup_{j = 1}^m R_{x_j} (B_{x_j} \times (-a_{x_j}, \infty))$ for all $k \geq k_0$.
\end{itemize}
\end{prop}

\begin{prop}\label{prop:35}
Let $U$, $(U_k)_{k \in \mathbb{N}}$ be as in the previous proposition and $\overline{U} \subset \Omega$ for a domain $\Omega$. Then
\begin{equation}
    \mathrm{dist}(\partial U_k,  \Omega^C) \rightarrow \mathrm{dist}(\partial U,\Omega^C) \quad (k \rightarrow \infty).
\end{equation}
In particular, there exists $k_0 \in \mathbb{N}$ such that for all $k \geq k_0$ one has $\overline{U_k}\subset \Omega$.
Furthermore there exists $L> 0$ such that $[\partial U_k]_{Lip} < L$  for all $k \in \mathbb{N}$ and for all $f \in C(\overline{\Omega})$ one has
\begin{equation}\label{eq:GWLippi}
   \lim_{k \rightarrow \infty} \int_{\partial U_k} f \; \mathrm{d}\mathcal{H}^{n-1}  = \int_{\partial U} f \; \mathrm{d}\mathcal{H}^{n-1}.
\end{equation}
\end{prop}
\begin{proof}
We use the notation from Proposition \ref{prop:dompert}.
We choose $x_1,...,x_m \in \partial U$ such that (for $k \geq k_0$) 
\begin{equation}
    \partial U \subset \bigcup_{j = 1}^m R_{x_j}(B_{x_j} \times (-a_{x_j}, \infty) ), \quad \partial U_k \subset \bigcup_{j = 1}^m R_{x_j}(B_{x_j} \times (-a_{x_j}, \infty) )  \quad \forall k \in \mathbb{N}.
\end{equation}
Now we can compute
\begin{align}
    \mathrm{dist}(\partial U_k, \Omega^C) & = \inf_{ z \in \Omega^C , x \in \partial U_k} |x-z| = \min_{j = 1,...,m} \inf_{z \in \Omega^C} \inf_{x \in \partial U_k \cap R_{x_j}(B_{x_j} \times (-a_{x_j}, \infty))} |x-z|\\
   & = \min_{j = 1,...,m} \inf_{z \in \Omega^C} \inf_{y' \in B_{x_j}} |R_{x_j} (y', \psi_k^{(x_j)}(y'))^T-z|
   \\ & =  \min_{j = 1,...,m} \inf_{z \in \Omega^C} \inf_{y' \in B_{x_j}} | (y', \psi_k^{(x_j)}(y'))^T-R_{x_j}^{-1} z|
   \\ &  \leq \min_{j = 1,...,m} \inf_{z \in \Omega^C} \inf_{y' \in B_{x_j}} | (y', \psi^{(x_j)}(y'))^T-R_{x_j}^{-1} z| + \max_{j = 1,...,m} ||\psi_k^{x_j} - \psi^{(x_j)}||_{L^\infty}
   \\ & = \mathrm{dist}(\partial U,  \Omega^C) + \max_{j = 1,...,m} ||\psi_k^{(x_j)} - \psi^{(x_j)}||_{L^\infty},
\end{align}
where the derivation of the last identity follows the lines of the first few steps in the computation, just backwards. Analogously one can show
\begin{align}
    \mathrm{dist}(\partial U_k, \Omega^C) \geq \mathrm{dist}(\partial U,  \Omega^C) - \max_{j = 1,...,m} ||\psi_k^{(x_j)} - \psi^{(x_j)}||_{L^\infty}.
\end{align}
We infer $\mathrm{dist}(\partial U_k, \Omega^C) \rightarrow \mathrm{dist}(\partial U, \Omega^C)$. Next we show that $[\partial U_k]_{Lip} < L$ for all $k \in \mathbb{N}$ for some $L$ independent of $k$. Indeed, by \eqref{eq:GammaLipschitzNorm} and the properties in Proposition \ref{prop:dompert} we obtain
\begin{equation}
    [\partial U_k]_{Lip} \leq \sum_{i = 1}^m \sqrt{1 + ||\nabla \psi_k^{(x_i)}||_{L^\infty(B_{x_i})}^2}   \leq m \sqrt{1+ M^2}. 
\end{equation}
The claim follows when taking $L := m \sqrt{1+M^2}$. It remains to show that for all $f \in C(\overline{\Omega})$ \eqref{eq:GWLippi} holds. To this end observe for $f \in C(\overline{\Omega})$ that 
\begin{align}
   \int_{\partial U_k} f \; \mathrm{d}\mathcal{H}^{n-1}  & =  \sum_{i = 1}^m \int_{B_{x_i}} f(R_{x_i}(y', \psi_k^{(x_i)}(y'))^T) \sqrt{1+|\nabla \psi_k^{(x_i)}|^2} \dy',\\
   \int_{\partial U} f \; \mathrm{d}\mathcal{H}^{n-1}  & =  \sum_{i = 1}^m \int_{B_{x_i}} f(R_{x_i}(y', \psi^{(x_i)}(y'))^T) \sqrt{1+|\nabla \psi^{(x_i)}|^2} \dy',
\end{align}
Next observe that for all $i = 1,...,m$ one has 
\begin{equation}\label{eq:convuniff}
    f(R_{x_i}( \cdot , \psi_k^{(x_i)}(\cdot) )  \rightarrow f(R_{x_i}( \cdot , \psi^{(x_i)} ( \cdot) ) ) \quad \textrm{ uniformly on $B_{x_i}$, \; $(k \rightarrow \infty)$.}
\end{equation}
 since $f$ is uniformly continuous and $\psi_k^{(x_i)} \rightarrow \psi^{(x_i)}$ uniformly. Also observe that 
\begin{equation}\label{eq:convpsi}
    \left\vert \sqrt{1+ |\nabla \psi_k^{(x_i)}|^2}- \sqrt{1+ |\nabla \psi^{(x_i)}|^2} \right\vert \leq |\nabla \psi_k^{(x_i)} - \nabla \psi^{(x_i)}| 
\end{equation}
since for all $z,w \in \mathbb{R}^n$ one has 
\begin{align}
   \left\vert  \sqrt{1+ |z|^2} - \sqrt{1+ |w|^2} \right\vert = \frac{||z|^2 - |w|^2|}{\sqrt{1+ |z|^2}+ \sqrt{1+ |w|^2}} = \frac{(|z| + |w|)||z|-|w||}{\sqrt{1+ |z|^2}+ \sqrt{1+ |w|^2}}  \leq ||z|- |w|| \leq |z-w| . 
\end{align}

Using \eqref{eq:convuniff} and \eqref{eq:convpsi} and the properties in Proposition \ref{prop:dompert} we find
\begin{align}
   &  \left\vert \int_{\partial U_k} f \; \mathrm{d}\mathcal{H}^{n-1} - \int_{\partial U} f \; \mathrm{d}\mathcal{H}^{n-1} \right \vert 
   \\ & =  \left\vert \sum_{i = 1}^m \int_{B_{x_i}} f(R_{x_i}(y', \psi_k^{(x_i)}(y'))^T) \sqrt{1+|\nabla \psi_k^{(x_i)}|^2} - f(R_{x_i}(y', \psi^{(x_i)}(y'))^T) \sqrt{1+|\nabla \psi^{(x_i)}|^2} \dy' \right\vert
   \\ & \leq \sum_{i = 1}^m \int_{B_{x_i}} \left\vert f( R_{x_i}(y', \psi_k^{(x_i)}(y'))^T)- f(R_{x_i}(y', \psi^{(x_i)}(y'))\right\vert  \sqrt{1+|\nabla \psi_k^{(x_i)}|^2} \dy' \\ & \quad + \sum_{i = 1}^m \int_{B_{x_i}} |f(R_{x_i}(y', \psi^{(x_i)}(y'))^T)| \left\vert  \sqrt{1+ |\nabla \psi_k^{(x_i)}|^2}- \sqrt{1+ |\nabla \psi^{(x_i)}|^2} \right\vert \dy' \\ & 
   \leq 
   \sum_{i =1}^m \sqrt{1+ M^2} || f( R_{x_i}(\cdot, \psi_k^{(x_i)}(\cdot))^T)- f(R_{x_i}(\cdot, \psi^{(x_i)}(\cdot))||_{L^\infty(B_{x_i})} 
   \\ & \quad + \sum_{i = 1}^m ||f||_{L^\infty(\overline{\Omega})} \int_{B_{x_i}} |\nabla \psi_k^{(x_i)} - \nabla \psi^{(x_i)}| \dy'  \quad \quad \quad \longrightarrow 0  \quad (k \rightarrow \infty),
\end{align}
where we used in the last step the uniform continuity of $f$ and the fact that by Proposition \ref{prop:dompert} $\nabla \psi_k^{(x_i)} \rightarrow \nabla \psi^{(x_i)}$ in $L^1(B_{x_i})$ for all $i = 1,...,m$. The claim follows.
\end{proof}

In the sequel we will also often localize and use Lipschitz domains that arise from the following procedures: 
\begin{prop}[{cf. \cite[Proposition 2.5.4]{Carbone}}]\label{prop:transverselip}
Let $U \subset \mathbb{R}^n$ be a $C^1$-domain and $z \in \partial U$. Then there exists $r_0 > 0$ such that for $r < r_0$ $U \cap B_r(z)$ is a Lipschitz domain. 
\end{prop}

\begin{prop}[Lipschitz subgraphs]\label{prop:Lipschitzsub}
Let $U \subset \mathbb{R}^{n-1}$ be open and bounded and $\Gamma= \{ (y, f(y) : y \in U \} \subset \mathbb{R}^n$ be a Lipschitz graph of a Lipschitz function $f: U \rightarrow \mathbb{R}$. Moreover let $\delta > 0$ and $B := U \times (- ||f||_\infty - \delta , ||f||_\infty + \delta).$ Then 
$
    D := \{(x,y) \in B :  y > f(x) \}
$
is a Lipschitz domain. 
\end{prop}
\begin{proof} 
The proof follows the lines of \cite[Proposition 2.5.4]{Carbone}.
\end{proof}

\section{Approximation and Compactness results}
\begin{prop}[Approximation in $L^\infty(\Gamma)$]\label{prop:L^infapp}
Let $\Gamma= \partial \Omega'$ be a boundary of a Lipschitz domain and $Q \in L^\infty(\Gamma)$. Then there exists a sequence $(Q_k)_{k \in \mathbb{N}} \subset C(\Gamma)$ such that $Q_k \rightarrow Q$ in $L^1(\Gamma)$ and $||Q_k||_{L^\infty(\Gamma)} \leq ||Q||_{L^\infty(\Gamma)}$. 
\end{prop}
\begin{proof}
First assume that 
$
    \Gamma = \{ (y ,f(y)) : y \in \overline{U} \}
$
is a Lipschitz graph of a function $f: U \rightarrow \mathbb{R}$ which is Lipschitz on an open rectangle ${U}$. Now we can define $\widetilde{Q}(x) := Q(x,f(x))$ for all $x \in U$ and observe that $\widetilde{Q} \in L^\infty(U).$ It is now possible to approximate $\widetilde{Q}$ in $L^1(U)$ by functions $(\widetilde{Q}_k)_{k \in \mathbb{N}} \subset C(\overline{U})$ by multiplying $\widetilde{Q}$  with cutoff functions and mollifying the result. Since cutting off and mollifying does not increase the $L^\infty(U)$-norm we can achieve
$
    ||\widetilde{Q}_k||_{L^\infty(U)}\leq ||\widetilde{Q}||_{L^\infty(U)} .
$
Next define $Q_k : \Gamma \rightarrow \mathbb{R}$ via
\begin{equation}
    Q_k(x, f(x)) := \widetilde{Q}_k(x) \quad (x \in \Gamma).
\end{equation}
One readily checks that $Q_k \in C(\Gamma)$. Moreover 
\begin{align}
    ||Q_k - Q||_{L^1(\Gamma)} & = \int_U |Q_k(x,f(x)) - Q(x,f(x)) | \sqrt{1+ |\nabla f(x)|^2} \dx \\ & \leq \sqrt{1+ ||\nabla f||^2_{L^\infty(U)}} ||\widetilde{Q}_k - \widetilde{Q}||_{L^1(U)} \rightarrow 0 \quad (k \rightarrow \infty).  
\end{align}
Our construction also reveals that if $Q$ has compact support in $\{ (y, f(y)) : y \in U \}$ then also $(Q_k)_{k \in \mathbb{N}}$ can be chosen to have compact support in the same set. This at hand, a standard argument involving a partition of unity implies the claim (cf. also Remark \ref{rem:partitionofunit}). 
\end{proof}

\begin{prop}[A compactness result in   $C^{0,1}(\overline{\Omega})$]\label{prop:C01approx}
Let $\Omega \subset \mathbb{R}^n$ be a smooth domain and $(v_k)_{k \in \mathbb{N}} \subset C^{0,1}( \overline{\Omega})$ such that $||v_k||_{C^{0,1}(\overline{\Omega})}$ is bounded. Then there exists a subsequence $(v_{l_k})_{k \in \mathbb{N}}$ and $v \in C^{0,1}(\overline{\Omega}) $ such that $v_{l_k} \rightarrow v$ uniformly on $\overline{\Omega}$ and 
\begin{equation}
    ||\nabla v||_{L^\infty( \Omega) } \leq \liminf_{k \rightarrow \infty} || \nabla v_{l_k}||_{L^\infty(\Omega)}.
\end{equation}
\end{prop}
\begin{proof}
Existence of the uniformly convergent subsequence $(v_{l_k})_{k \in \mathbb{N}}$ is a direct consequence of the Arzela-Ascoli theorem. Let $v \in C^{0}(\overline{\Omega})$ be the limit.  The fact that $v \in C^{0,1}(\overline{\Omega})$ follows from 
\begin{equation}
    |v_{l_k}(x) - v_{l_k}(y)| \leq ||v_{l_k}||_{C^{0,1}(\overline{\Omega})} |x-y|,
\end{equation}
where we used that for all $u \in C^{0,1}(\overline{\Omega})$ 
\begin{equation}
    ||u||_{C^{0,1}(\overline{\Omega})} := ||u||_{C^0(\overline{\Omega})} + \sup_{x,y \in \overline{\Omega},x \neq y} \frac{|u(x)-u(y)|}{|x-y|}.
\end{equation}
The lower semicontinuity of $u \mapsto ||\nabla u||_{L^\infty}$ with respect to uniform convergence can be seen in many ways. One possibility is the following computation.
\begin{align}
   & ||\nabla u||_{L^\infty(\Omega)}  = \sup_{\phi \in L^1(\Omega; \mathbb{R}^n),||\phi||_{L^1}\leq 1} \int_\Omega \nabla u \cdot \phi \dx \\ & =  \sup_{\phi \in C_0^\infty(\Omega; \mathbb{R}^n), ||\phi||_{L^1} \leq 1} \int_\Omega \nabla u \cdot \phi \dx = \sup_{\phi \in C_0^\infty(\Omega; \mathbb{R}^n),||\phi||_{L^1} \leq 1} \int_\Omega u \cdot \mathrm{div}(\phi) \dx.
\end{align}
\end{proof}


\section{A maximum principle for $BV$-solutions}\label{app:maxpr}

In this section we discuss a generalized maximum principle, which allows us to estimate the $L^\infty$ norm of a harmonic function in terms of its \emph{$BV$-trace}, whose definition and properties will also be recalled in this section.
\begin{lemma}[{Traces of $BV$-Functions, cf. \cite[Theorem 2.10]{Giusti}}]\label{lem:BVtr}
Suppose that $U \subset \mathbb{R}^n$ has Lipschitz boundary and $f \in BV(U)$. Then there exists some $g \in L^1(\partial U, \mathcal{H}^{n-1})$ such that for $\mathcal{H}^{n-1}$ a.e. $x \in \partial U$ one has
\begin{equation}
\lim_{ r \rightarrow 0 } \fint_{B_r(x) \cap U} f(y) \dy  = g(x),
 \quad 
\lim_{ r \rightarrow 0 } \fint_{B_r(x) \cap U} |f(y)-g(x)| \dy = 0.
\end{equation}
Further, for all $\phi \in C^1_c(\mathbb{R}^n;\mathbb{R}^n)$ one has 
\begin{equation}
\int_U  f \;  \mathrm{div}(\phi) \dx  = - \int_U \phi \; \mathrm{d}[Df] + \int_{\partial U} g \; \phi \cdot \nu \; \mathrm{d}\mathcal{H}^{n-1}.
\end{equation}
The function $g$ is called \emph{$BV$-trace} of $f$.
\end{lemma}

This at hand we can finally prove the desired maximum principle.

\begin{proof}[Proof of Lemma \ref{lem:maxpr}] Let $U$,$w$ be as in the statement. For the rest of the proof let  $(\psi_\epsilon)_{\epsilon> 0}$ be the standard mollifier. \\
\textbf{Step 1:} $W^{1,1}$-regularity of $w$.
First of all note that $w \in C^\infty(U)$. This can for example be shown by inferring from \eqref{eq:weakharm} that  $w* \psi_\epsilon$ is harmonic (in the classical sense) in $U^\epsilon:= \{x \in U : \mathrm{dist}(x, U^C) > \epsilon \}$. This implies also that $w * \psi_\epsilon$ has the mean-value property, i.e. 
\begin{equation}
w * \psi_\epsilon(x) = \fint_{ B_r(x)} (w* \psi_\epsilon)(y) \dy \quad \forall x \in U^\epsilon, r < \mathrm{dist}(x, (U^{\epsilon})^C)).
\end{equation}
Thereupon one can pass to the limit as $\epsilon \rightarrow 0 $ and infer that $w$ must have the mean value property in the whole of $U$. This implies by standard computations $w \in C^\infty(U)$ and $w* \psi_\epsilon = w$. We infer that $w \in C^\infty(U) \cap BV(U)$. We claim next that $w \in W^{1,1}(U)$. Indeed: For all $\phi \in C_0^\infty(U)$ one has since $w \in C^\infty(U)$
\begin{equation}
\int_U w \partial_i \phi \dx  =  - \int_{U} ( \partial_i w ) \phi \dx, 
\end{equation}
but on the other hand 
\begin{equation}
\int_{U} w \partial_i  \phi \dx  = - \int_{U} \phi \;  \mathrm{d}[Dw]^i.
\end{equation}
Hence $[Dw]^i = (\partial_i w) \mathcal{L}^n$, where $\mathcal{L}^n$ denotes the $n$-dimensional Lebesgue measure. Since $|[Dw]|(U)< \infty$ we conclude 
\begin{equation}
\infty > |[Dw]^i|(U) = \int_U |\partial_i w | \dx . 
\end{equation} 
The $W^{1,1}$-regularity follows. Note however that the maximum principle is also not clear for $W^{1,1}$-solutions.\\
\textbf{Step 2}: We find an equation that characterizes $w$ uniquely.
To this end let $\phi \in C^2(\overline{U})$ be such that $\phi\vert_{\partial U} = 0$. Our aim is to show that 
\begin{equation}
\int_U  w \Delta \phi \dx = \int_{\partial U}  \mathrm{tr}(w) \partial_{\nu} \phi \; \mathrm{d} \mathcal{H}^{n-1}. 
\end{equation}
To this end we extend $\phi$ to a $C^2_c(\mathbb{R}^n)$-function which is possible due to the boundary regularity. Let $\overline{\phi} \in C^2_c(\mathbb{R}^n)$ be this extension. Then by Lemma \ref{lem:BVtr} one has 
\begin{align}
\int_U w \Delta \phi \dx & = \int_{U} w  \Delta \overline{\phi} \dx = \int_U w  \; \mathrm{div}( \nabla \overline{\phi} ) \; \mathrm{d}x 
\\ & = \int_{\partial U} \mathrm{tr}(w)  \nabla \phi \cdot \nu \; \mathrm{d}\mathcal{H}^{n-1}  - \int_U \nabla \phi  \; \mathrm{d} [Dw] 
\\  & = \int_{\partial U} \mathrm{tr}(w)  \nabla \phi \cdot \nu \; \mathrm{d}\mathcal{H}^{n-1}  - \int_U \nabla \phi \nabla w \dx . \label{eq:PI1} 
\end{align}
Next we intend to approximate $\phi$ by a sequence of functions $(\phi_n)_{n \in \mathbb{N}} \subset C_0^\infty(U)$ to get rid of the second summand.
This approximation is nonstandard because of the fact that one only has $\nabla w \in L^1$. We claim that there exist $(\phi_n)_{n \in \mathbb{N}} \subset C_0^\infty(U)$ such that $\phi_n \rightarrow \phi$ in $W_0^{1,2}(U)$ and $||\nabla \phi_n||_{L^\infty} \leq ||\nabla \phi||_{L^\infty}$. To this end notice that (since $\phi\vert_{\partial \Omega} = 0$) for all $n \in \mathbb{N}$ the functions  $\left( \phi - \frac{1}{n} \right)^+$ and $\left( \phi + \frac{1}{n} \right)^-$ have compact support  in $U$ and lie in $W^{1,2}_0(U)$. This implies that for $\epsilon < \epsilon_0$ appropriately small $\left( \phi - \frac{1}{n} \right)^+ * \psi_\epsilon \in C_0^\infty(U)$ and $\left( \phi - \frac{1}{n} \right)^+ * \psi_\epsilon  \rightarrow \left( \phi - \frac{1}{n} \right)^+$  in $W_0^{1,2}(U)$ as $\epsilon \rightarrow 0$.
We infer that for all $n \in \mathbb{N}$ there exists $\epsilon_n > 0$ such that $ \left( \phi - \frac{1}{n} \right)^+ * \psi_{\epsilon_n} \in C_0^\infty(U)$ and $||\left( \phi - \frac{1}{n} \right)^+ * \psi_{\epsilon_n} - \left( \phi - \frac{1}{n} \right)^+ ||_{W_0^{1,2}} < \frac{1}{n}$. Similarly we can find $\widetilde{\epsilon_n} > 0$ such that  $ \left( \phi - \frac{1}{n} \right)^- * \psi_{\widetilde{\epsilon}_n} \in C_0^\infty(U)$ and $||\left( \phi - \frac{1}{n} \right)^- * \psi_{\widetilde{\epsilon}_n} - \left( \phi - \frac{1}{n} \right)^- ||_{W_0^{1,2}} < \frac{1}{n}$. By possibly shrinking $\epsilon_n$ or $\widetilde{\epsilon}_n$ we may assume that $\epsilon_n = \widetilde{\epsilon_n}$.
 Now define 
\begin{equation}
\phi_n := \left( \phi - \frac{1}{n} \right)^+ * \psi_{\epsilon_n} + \left( \phi + \frac{1}{n} \right)^- * \psi_{{\epsilon}_n} 
\end{equation}
It is then straightforward to show that $(\phi_n)_{n\in \mathbb{N}} \subset C_0^\infty(U)$ and $\phi_n \rightarrow \phi$ in $W_0^{1,2}(U)$. It remains to show that $||\nabla \phi_n||_{L^\infty} \leq ||\nabla \phi||_{L^\infty}$. To this end notice that 
\begin{align}
\nabla \left[ \left( \phi - \frac{1}{n} \right)^+ * \psi_{\epsilon_n} \right] (x) & = \int_U \left( \phi - \frac{1}{n} \right)^+(y) \nabla_x \psi_{\epsilon_n}(x-y) \dy  \\ &  = - \int_U \left( \phi - \frac{1}{n} \right)^+(y) \nabla_y \psi_{\epsilon_n}(x-y)  \dy 
\\ & = \int_U \nabla \left( \phi - \frac{1}{n} \right)^+(y)  \psi_{\epsilon_n}(x-y)  \dy
\\ & = \int_{ U \cap \{\phi > \frac{1}{n}\}} (\nabla \phi)  \psi_{\epsilon_n}(x-y) \dy.
\end{align}
 Similarly 
 \begin{equation}
 \nabla \left[ \left( \phi - \frac{1}{n} \right)^- * \psi_{\epsilon_n} \right] (x) = \int_{ U \cap \{\phi < - \frac{1}{n}\}} (\nabla \phi) \;  \psi_{\epsilon_n}(x-y) \dy.
 \end{equation}
 We obtain by Young's inequality that 
 \begin{equation}
 |\nabla \phi_n(x)| =  \left\vert \int_{ U \cap \{ |\phi| > \frac{1}{n} \} } (\nabla \phi) \;  \psi_{\epsilon_n}(x-y) \dy  \right\vert \leq ||\nabla \phi||_{L^\infty(U)} ||\psi_{\epsilon_n}||_{L^1(\mathbb{R}^n)}  = ||\nabla \phi||_{L^\infty(U)}. 
 \end{equation}
 It follows that $||\nabla \phi_n||_{L^\infty} \leq ||\nabla \phi||_{L^\infty}$. Note that we may assume after the choice of an appropriate subsequence that $\nabla \phi_n \rightarrow \nabla \phi$ pointwise almost everywhere. Since $||\nabla \phi_n||_{L^\infty} \leq ||\nabla \phi||_{L^\infty}$ and $|\nabla w | \; ||\nabla \phi||_{L^\infty}  \in L^1$ we infer by the dominated convergence theorem that 
 \begin{equation}
 \int_U \nabla w \nabla \phi \dx = \lim_{n \rightarrow \infty} \int_U \nabla w \nabla \phi_n \dx = \lim_{n \rightarrow \infty }  \int_U  w \Delta \phi_n  = 0 ,
\end{equation}  
where we used \eqref{eq:weakharm} in the last step. We infer from \eqref{eq:PI1} that for each $\phi \in C^2(\overline{U})$  such that $\phi\vert_{\partial U }  = 0$ one has 
\begin{equation}\label{eq:realeqw}
\int_U w \Delta \phi \dx  = \int_{\partial U } \mathrm{tr}(w) \partial_{\nu} \phi \; \mathrm{d}\mathcal{H}^{n-1}. 
\end{equation}
With this equation we will be able to show that $||w||_{L^\infty(U)} \leq || \mathrm{tr}(w)||_{L^\infty( \partial U ) }$. \\
\textbf{Step 3:} Show the claim under the additional assumption that $\mathrm{tr}(w) $ has a $C^\infty$-extension on $\overline{U}$, say $\overline{w} \in C^\infty(\overline{U})$ is such that $\overline{w}\vert_{\partial U} = \mathrm{tr}(w)$. Then one has  for all $\phi \in C^2(\overline{U})$ such that $\phi \vert_{\partial U}  = 0$ by the divergence theorem
\begin{align}
\int_U w \Delta \phi \dx =  \int_{\partial U } \mathrm{tr}(w) \partial_{\nu} \phi \; \mathrm{d}\mathcal{H}^{n-1} & = \int_{\partial U } \overline{w} \partial_{\nu} \phi \; \mathrm{d}\mathcal{H}^{n-1} \\ & = \int_{ U} \nabla \overline{w} \nabla \phi \dx +  \int_U \overline{w} \Delta \phi  \dx 
\\ &=-  \int_U \Delta \overline{w} \phi \dx  + \int_U \overline{w} \Delta \phi  \dx 
\end{align}
Hence for all $\phi \in C^2(\overline{U})$ such that $\phi\vert_{\partial U} = 0$ one has 
\begin{equation}
\int_U (w - \overline{w}) \Delta \phi \dx =  - \int_U \Delta \overline{w} \phi \dx . 
\end{equation}
By Lemma \ref{lem:ponci} we conclude that $w - \overline{w} \in W_0^{1,q}(U)$ for some $q > 1$ and solves $\Delta (w- \overline{w}) = - \Delta \overline{w}$ in the weak $W_0^{1,q}$-sense. By elliptic regularity this implies that $w - \overline{w} \in C^\infty(\overline{U})$. We infer that $w \in C^\infty( \overline{U})$, $\Delta w = 0$ classically, and $w \vert_{\partial U} = \overline{w} \vert_{\partial U} = \mathrm{tr}(w)$. 
By the (classical) maximum principle we have $||w||_{L^\infty(U)} \leq ||\mathrm{tr}(w)||_{L^\infty(\partial U)}$. We have shown the claim under the additional assumption of Step 3.\\
\textbf{Step 4:} We show the claim under the (weaker) additional assumption that $\mathrm{tr}(w) \in C^0(\partial U)$. If $\mathrm{tr}(w) \in C^0(\partial U)$, then by Tietze's extension theorem we can find $v \in C_0^0(\mathbb{R}^n)$ such that $v= w $ on $\partial U$. Now $v$ can be locally uniformly approximated by $(v_n)_{n \in \mathbb{N}} \subset C_0^\infty( \mathbb{R}^n)$ which means that $w$ can be uniformly approximated on $\partial U$ by $v_n \vert_{\partial U}$ --- i.e. by functions that possess a $C^\infty$ extension to $\overline{U}$. Now let $(w_n)_{n \in \mathbb{N}} \subset C^\infty(\overline{U}) $ be the solutions of 
\begin{equation}
\begin{cases} 
-\Delta w_n = 0  & \mathrm{in } \;  U, \\ \quad \; \;  w_n = v_n &\mathrm{on} \;  \partial U.
\end{cases} 
\end{equation}
Note that then  $||w_n - w_m||_{L^\infty(U)} \leq ||v_n - v_m ||_{L^\infty(\partial U)}$ and hence $w_n$ converges in $L^\infty(U)$ to some $\widetilde{w} \in L^\infty(U)$. Using  the Gauss divergence theorem and \eqref{eq:realeqw}  we find for each $\phi \in C^2(\overline{U})$ such that $\phi\vert_{\partial U } = 0$ that 
\begin{align}
\int_U \widetilde{w} \Delta \phi \dx & = \lim_{n \rightarrow \infty} \int_U w_n \Delta \phi \dx = \lim_{n \rightarrow \infty} \int_{\partial U} v_n \partial_\nu \phi \; \mathrm{d}\mathcal{H}^{n-1} 
\\ & =  \int_{\partial U} \mathrm{tr}(w) \partial_\nu \phi \; \mathrm{d} \mathcal{H}^{n-1} = \int_U w \Delta \phi \dx .
\end{align}
Hence we infer that for all $\phi \in C^2(\overline{U})$ such that $\phi\vert_{\partial U} = 0$ one has
\begin{equation}
\int_U (\widetilde{w} - w ) \Delta \phi \dx = 0 .
\end{equation}
By Lemma \ref{lem:ponci} we infer that $\widetilde{w} - w$ lies in $W_0^{1,q}(\Omega)$ for some $q > 1$ and is a weak solution of $\Delta( \widetilde{w} - w) = 0$. By uniqueness of weak solutions of the Dirichlet problem in $W_0^{1,q}$ we find $\widetilde{w} = w$. Hence 
\begin{equation}
||w||_{L^\infty(U)} = || \widetilde{w} ||_{L^\infty(U)} = \lim_{n \rightarrow \infty} ||w_n||_{L^\infty(U)} \leq \liminf_{n \rightarrow \infty} ||v_n ||_{L^\infty(\partial U)} = ||\mathrm{tr}(w)||_{L^\infty(\partial U)}.  
\end{equation}  
We have now shown the claim under the additional assumption of Step 4.\\
\textbf{Step 5:} Next we assume nothing except $\mathrm{tr}(w) \in L^\infty(\partial U)$. We know that then by Proposition \ref{prop:L^infapp} there exist $(v_n)_{n \in \mathbb{N}} \subset C^0(\partial U)$ such that $v_n \rightarrow \mathrm{tr}(w)$ in $L^1(\partial U)$ and $||v_n||_{L^\infty(\partial U)} \leq ||w||_{L^\infty(\partial U)}$. Now let $(w_n)_{n \in \mathbb{N}} \subset C^\infty( \overline{U})$ be the weak $W_0^{1,2}$-solutions to   
\begin{equation}
\begin{cases} 
-\Delta w_n = 0  & \mathrm{in } \;  U, \\ \quad \; \;  w_n = v_n &\mathrm{on} \;  \partial U.
\end{cases} 
\end{equation}
Then by \cite[p.467]{Anna} one has for each $x \in U$ 
\begin{equation}
w_n(x) = \int_{\partial U} K_U (x,y) v_n(y) \; \mathrm{d}y ,
\end{equation}
where $K_U \in C^\infty( U \times \partial U)$ is the Poisson kernel for $U$. Since by \cite[Theorem 4]{Anna} 
\begin{equation}
|K_U(x,y)| \leq C_U \frac{\mathrm{dist}(x, \partial U)}{|x-y|^{n-1}},
\end{equation}
the dominated convergence theorem yields for all $x \in U$  
\begin{equation}
\lim_{n \rightarrow \infty} w_n(x) = \int_{\partial U}  K_U(x,y) \mathrm{tr}(w)(y) \; \mathrm{d}y = : w^*(x), \quad (x \in U). 
\end{equation}
Hence $(w_n)_{n \in \mathbb{N}}$ converges pointwise on $U$ to some measurable function $w^*: U \rightarrow \mathbb{R}$. Since $||w_n||_{L^\infty(U)} \leq ||v_n||_{L^\infty(\partial U)} \leq ||\mathrm{tr}(w)||_{L^\infty(\partial U)} $ we infer that $w^* \in L^\infty$ and
\begin{equation}
||w^*||_{L^\infty} \leq \liminf_{n \rightarrow \infty} || w_n||_{L^\infty} \leq  ||\mathrm{tr}(w)||_{L^\infty(\partial U)}.
\end{equation}
By the dominated convergence theorem, the Gauss divergence theorem and \eqref{eq:realeqw} we infer for all $\phi \in C^2(\overline{U})$ such that $\phi\vert_{\partial U} = 0$ 
\begin{align}
\int_U w^* \Delta \phi \dx & = \lim_{n \rightarrow \infty} \int_U w_n \Delta \phi \dx = \lim_{n \rightarrow \infty} \int_{\partial U} v_n \partial_\nu \phi \; \mathrm{d}\mathcal{H}^{n-1} 
\\ & =  \int_{\partial U} \mathrm{tr}(w) \partial_\nu \phi \; \mathrm{d} \mathcal{H}^{n-1} = \int_U w \Delta \phi \dx .
\end{align}
It follows again that 
\begin{equation}
\int_U ( w^* - w) \Delta \phi \dx  = 0 \quad \forall \phi \in C^2(\overline{U}) : \phi\vert_{\partial U} = 0. 
\end{equation}
and hence by \ref{lem:ponci} we obtain that $w^*- w \in W_0^{1,q}(U)$ for some $q >1$ and by uniqueness of solutions to the Dirichlet problem in $W_0^{1,q}(U)$ we infer once again $w = w^*$ which gives $||w||_{L^\infty(U)} \leq ||w^*||_{L^\infty(U)} \leq || \mathrm{tr}(w)||_{L^\infty(\partial U)}$. 
\end{proof}



\end{document}